\documentclass[final,1p,times]{elsarticle}

\journal{Topology and its Applications}

\usepackage[utf8]{inputenc}
\usepackage[russian,english]{babel}

\usepackage{amsthm,amsmath,amssymb,eucal,mathrsfs}
\usepackage{tikz-cd,url}
\usepackage{mathabx} 
\usepackage[colorlinks]{hyperref} 

\newtheorem{thm}{Theorem}[section]
\newtheorem{cor}[thm]{Corollary}
\newtheorem{lem}[thm]{Lemma}
\newtheorem{prop}[thm]{Proposition}
\theoremstyle{definition}
\newtheorem{defn}[thm]{Definition}
\newtheorem{rem}[thm]{Remark}
\newtheorem{exm}[thm]{Example}
\numberwithin{equation}{section}

\DeclareMathAlphabet{\matheurm}{U}{eur}{m}{n}
\newcommand{\cpt}{\mathrm{cpt}} 
\newcommand{\lrg}{\mathrm{lrg}} 

\newcommand{\sB}{\mathscr{B}}
\newcommand{\sC}{\mathscr{C}}
\newcommand{\sE}{\mathscr{E}}

\renewcommand{\C}{\mathbb{C}}
\newcommand{\N}{\mathbb{N}}
\newcommand{\Q}{\mathbb{Q}}
\newcommand{\R}{\mathbb{R}}
\newcommand{\Z}{\mathbb{Z}}

\renewcommand{\G}{\Gamma}

\newcommand{\cC}{\mathcal{C}}
\newcommand{\cF}{\mathcal{F}}
\newcommand{\cM}{\mathcal{M}}
\newcommand{\cP}{\mathcal{P}}
\newcommand{\cT}{\mathcal{T}}
\newcommand{\cU}{\mathcal{U}}
\newcommand{\cV}{\mathcal{V}}

\newcommand{\Aut}{\mathrm{Aut}}
\newcommand{\Homeo}{\mathrm{Homeo}}
\newcommand{\id}{\mathrm{id}}
\newcommand{\proj}{\mathrm{proj}}
\newcommand{\TOP}{\mathrm{TOP}}

\newcommand{\bdry}{\partial}
\newcommand{\eps}{\varepsilon}
\newcommand{\iso}{\cong}
\newcommand{\longra}{\longrightarrow}
\newcommand{\x}{\times}
\newcommand{\vphi}{\varphi}

\newcommand{\ol}[1]{\overline{#1}}

\newcommand{\inv}{^{-1}}

\newcommand{\card}{\mathrm{card}}
\newcommand{\cB}{\mathcal{B}}
\newcommand{\kc}{\mathfrak{c}}
\newcommand{\bI}{\mathbb{I}}
\newcommand{\pr}{\mathrm{pr}}
\newcommand{\gens}[1]{\langle #1 \rangle}
\newcommand{\DS}{\displaystyle}
\newcommand{\bs}{\backslash}
\newcommand{\cR}{\mathcal{R}}
\newcommand{\bE}{\mathbb{E}}
\newcommand{\bB}{\mathbb{B}}
\newcommand{\cK}{\mathcal{K}}

\newcommand{\bigbcoast}{\mathop{\boxed{\bigcoast}}}

\begin{document}

\begin{frontmatter}
 
\title{Cardinal-indexed classifying spaces for families\\ of subgroups of any topological group}
\tnotetext[dedication]{This paper is dedicated to Sergey Antonyan on the occasion of his 65th birthday.}

\author{Qayum Khan}

\address{Department of Mathematics \hfill Indiana University \hfill Bloomington IN 47405 USA}
\ead{qkhan@indiana.edu}
 
\begin{abstract}
For $G$ a topological group, existence theorems by Milnor (1956), Gelfand--Fuks (1968), and Segal (1975) of classifying spaces for principal $G$-bundles are generalized to $G$-spaces with torsion.
Namely, any $G$-space approximately covered by tubes (a generalization of local trivialization) is the pullback of a universal space indexed by the orbit types of tubes and cardinality of the cover.
For $G$ a Lie group, via a metric model we generalize the corresponding uniqueness theorem by Palais (1960) and Bredon (1972) for compact $G$.
Namely, the $G$-homeomorphism types of proper $G$-spaces over a metric space correspond to stratified-homotopy classes of orbit classifying maps.

The former existence result is enabled by Segal's clever but esoteric use of non-Hausdorff spaces.
The latter uniqueness result is enabled by our own development of equivariant ANR theory for noncompact Lie $G$.
Applications include the existence part of classification for unstructured fiber bundles with locally compact Hausdorff fiber and with locally connected base or fiber, as well as for equivariant principal bundles which in certain cases via other models is due to Lashof--May (1986) and to L\"uck--Uribe (2014).
From a categorical perspective, our general model $E_\cF^\kappa G$ is a final object inspired by the formulation of the Baum--Connes conjecture (1994).
\end{abstract}


\end{frontmatter}

\section*{Introduction}

Let $G$ be a topological group.
Throughout this paper, we mostly consider right $G$-spaces $X$.
Write $X/G := \{ xG ~|~ x \in X \}$ for its \emph{orbit space,} where $xG := \{ xg ~|~ g \in G \}$ denotes an \emph{orbit.}
An \emph{isotropy group} is the subgroup $G_x := \{ g \in G ~|~ xg = x \}$.
For any subgroup $H$ of $G$, consider the \emph{right $H$-cosets} $Hg := \{ hg ~|~ h \in H \}$ and endow the right $G$-set $H\bs G := \{ Hg ~|~ g \in G \}$ with the quotient topology.
The \emph{balanced product} of a right $G$-space $X$ and a left $G$-space $Y$ is
\[
X \x_G Y ~:=~ (X \x Y) ~/~ ( (xg,y) \sim (x,gy) ).
\]

Historically, several models of classifying spaces for principal $G$-bundles exist.
Milnor (1956) introduced $\bE G$ over base spaces $B$ that are paracompact Hausdorff, and Dold (1963) proved a uniqueness theorem in which homotopy classes of maps $B \longra \bB G := \bE G / G$ correspond to isomorphism clases of principal bundles over $B$.
Gelfand--Fuks (1968) generalized this to all base spaces $B$ that are Tikhonov (that is, completely regular Hausdorff) via unnormalized joins, and Segal (1975) used non-Hausdorff cones to further observe a model $EG$ that works for $B$ being any topological space of arbitrary weight; however both are at the loss of uniqueness.

We extend Segal's model to a so-called $E_\cF G$ to allow for fixed points with isotropy conjugate into a given set $\cF$ of subgroups of $G$, where local triviality $U \x G$ of a principal bundle is replaced by the notion of an \emph{$H$-tube} $S \x_H G$, that is, the induction of a right $H$-space $S$ (called a slice) to a right $G$-space.
Consequently, other models over more specialized base spaces admit a canonical $G$-map to this one, that is, it is more universal for $B$ any topological space.
To avoid set-theoretic paradoxes, it is limited to $E_\cF^\kappa G$ by the \emph{weight $wB \leqslant \kappa$ of $B$}, which is the minimum cardinality for a base of the topology; for example, second-countable spaces have weight the first infinite cardinal $\aleph_0$.
If the covering of the $G$-space $X$ by tubes is \emph{approximate} and $\cF$ consists of closed subgroups, then our existence theorem (\ref{thm:classify}) is that $X$ is the pullback of $E_\cF^\kappa G$ along a map $X/G \longra E_\cF^\kappa G/G$ whose explicit formula is canonically determined by the tube data with $\kappa = w(X/G)$.

Now assume $G$ is an arbitrary Lie group, such as a countable discrete group.
Following Bredon's improvement (1972) of Palais' argument (1960) for compact Lie $G$, as well as employing and developing (\ref{thm:extensor}) modern advances in the equivariant theory of absolute neighborhood retracts (ANRs), we prove the following uniqueness theorem (\ref{thm:unique}).
Suppose that $X$ admits a $G$-invariant metric and that the action of $G$ is \emph{proper in the sense of Palais} (1961), so an \emph{isovariant} ($H=G_x$) covering by tubes exists by Palais' slice theorem.
Then the isomorphism classes of such $X$ over a given base space $B = X/G$ bijectively correspond to stratified-homotopy classes of maps $B \longra \bB_\cF G$, where $B$ has the induced orbit-type stratification and $\cF$ is a set of compact subgroups of $G$ containing all the isotropy groups of $X$ without any conjugate representatives.
Here $\bB_\cF G$ is the orbit space of the right $G$-space $\bE_\cF G$, which is our unnormalized join inspired by Gelfand--Fuks with coarse cones instead of fine ones.
Palais--Bredon assume that $\cF$ is finite so use normalized (Milnor) joins, as well as assume that $B$ is finite-dimensional; we remove these cardinal limitations, as well as no longer assume $G$ is compact whereby proper was automatic.

\section{Preliminaries}

M~McCord introduced the following notion as non-Hausdorff cone \cite[\S8]{McCord}.

\begin{defn}[McCord]
Recall Sierpi\'{n}ski\footnote{Sierpi\'{n}ski \cite[\S3, \S9]{Sierpinski_book2eng} noted $\bI_1$ is the nondiscrete nonindiscrete Fr\'{e}chet $\cV$-space \cite[V]{Frechet} on two points. Open sets of a topological space correspond bijectively to continuous functions~to~$\bI_1$.} space $\bI_1 := (\{0,1\}, \{\emptyset, \{1\}, \{0,1\}\})$.
Let $(A,\cT)$ be a topological space.
Write $A_+ := A\sqcup\{0\}$.
The \textbf{indiscrete cone} is
\[
c(A,\cT) ~:=~ (A,\cT)_+ \wedge \bI_1
~=~ (A_+, \cT\cup\{A_+\}).
\]
\end{defn}

This let G~Segal \cite{Segal} cleverly simplify a construction of Gelfand--Fuks \cite{GF}.
We generalize it here to non-free actions for which Segal's construction is $\cF = \{1\}$.

\begin{defn}\label{defn:Segal}
Let $G$ be a topological group.
Let $\cF$ be any set of subgroups of $G$.
Define $\cF \bs G := \bigsqcup_{H \in \cF} H \bs G$ equipped with the coherent topology.
Let $\kappa$ be a cardinal.
Write $\kappa = \card(I)$ for a set $I$.
Using the product and subspace topologies, define
\[
E_\cF^\kappa G ~:=~ (c(\cF \bs G))^I - \{0\}^I
\]
and the \textbf{$\kappa$-indexed $\cF$-classifying space} $B_\cF^\kappa G := E_\cF^\kappa G / G$ with quotient topology.
Note the $G$-homeomorphism type of $E_\cF^\kappa G$ does not depend on the representative $I$.
For less cumbersome reading, we abbreviate $E^\kappa G := E_{\{1\}}^\kappa G$ and $B^\kappa G := B_{\{1\}}^\kappa G$.
\end{defn}

We expand Biller's \cite[2.1]{Biller_proper} beyond Hausdorff $G$ and $X$ and compact $H \in \cF$.
Our expanded definition here also adds the notions of approximate and isovariant.

\begin{defn}\label{defn:isocover}
Let $G$ be a topological group.
Let $X$ be a topological $G$-space.
For any $x \in X$, its \textbf{isotropy group} is $G_x := \{g \in G ~|~ xg=x \}$.
Let $\cF$ be any set of subgroups of $G$.
An open $G$-subset $T$ of $X$ is an \textbf{$\cF$-tube} if it is $G$-homeomorphic to $S \x_H G$ for some $H \in \cF$ and $H$-space $S$.
We say that $X$ is \textbf{covered by $\cF$-tubes} if $X = \bigcup_{i \in I} T_i$ for some $\cF$-tubes $\{T_i \approx S_i \x_{H_i} G \}_{i \in I}$.
More specifically, the cover is \textbf{approximate} if, for each point $x \in X$ and neighborhood $O$ of $G_x$ in $G$, there exist $i \in I$ and $g \in G$ with $x \in T_i$ and $G_x \leqslant g\inv H_i g \subset O$.
In particular, the cover is \textbf{isovariant} if each $x \in X$ admits some $i \in I$ with $x \in T_i$ and $G_x$ conjugate to~$H_i$.
\end{defn}

(Isovariant) covering by $\cF$-tubes implies tomDieck's ``(strongly) locally $\cF$-trivial'' \cite[p46]{tomDieck_TG}.

\begin{prop}\label{prop:cover_classifying}
Any $E_\cF^\kappa G$ can be covered by $\cF$-tubes, in fact, by $\kappa \cdot \card(\cF)$-many.
Also, the cover by $\cF$-tubes is isovariant when restricted to the following $G$-subset:
\[
\sE_\cF^\kappa G ~:=~ \left\{ e \in E_\cF^\kappa G ~|~ \exists i \in I : G_e = G_{e_i} \right\}.
\]
\end{prop}

We shall call $\sB_\cF^\kappa G := \sE_\cF^\kappa G / G$ the \textbf{isovariant $\kappa$-indexed $\cF$-classifying space}.
The $G$-space $\sE_\cF^\kappa G$ is analogous to Palais' reduced join \cite[1.3.6]{Palais_book} \cite[p108]{Bredon_TG}.
Note $\sE_\cF^\kappa G$ is dense in $E_\cF^\kappa G$ if $\cF$ is closed under conjugacy and $\kappa$-fold intersections.

\begin{proof}
Fix $i \in I$ and $H \in \cF$.
Define a $G$-space $T(i,H)$ and $H$-subspace $S(i,H)$ by
\begin{eqnarray*}
T(i,H) &:=& \{ e \in E_\cF^\kappa G ~|~  e_i \in H \bs G \}\\
S(i,H) &:=& \{ e \in E_\cF^\kappa G ~|~  e_i = H \in H \bs G \}.
\end{eqnarray*}
It remains to show that the following canonical bijective $G$-map is an open function:
\begin{equation}\label{eqn:tube}
\mu: S(i,H) \x_H G \longra T(i,H) \quad;\quad [e,g] \longmapsto eg = (e_i g)_{i \in I}.
\end{equation}

Let $O$ be open in $G$.
Let $j \neq i \in I$ and $K \in \cF$.
Let $U$ be open in $K \bs G$.
Consider
\[
B(j,U) ~:=~ \{ e \in S(i,H) ~|~ e_j \in U \}.
\]
For any $k \in I$, write $\pi_k: (c(\cF \bs G))^I \longra c(\cF \bs G)$ for the $k$-th projection $(e \mapsto e_k)$ and $V_k : = (\pi_k \circ \mu)[B(j,U) \x O]$.
Since $O$ is open $G$, note $V_i = \{H\} O = \{H\} (\bigcup_{h \in H} h O)$ is open in $H \bs G$.
Since $U$ is open in $K \bs G$, note $V_j = U O = \bigcup_{g \in O} U g$ is open in $K \bs G$ \cite[I:3.1i]{tomDieck_TG}.
Otherwise $V_k = c(\cF \bs G)$ is open in $c(\cF \bs G)$ for all $k \neq i,j$.
Thus $\mu[B(j,U) \x O] \subseteq T(i,H)$ is open with respect to the product topology of $(c(\cF \bs G))^I$.

Observe that the subspace topology of $S(i,H)$ has subbase
\[
\cB ~:=~ \left\{ B(j,U) ~|~ j \in I-\{i\} \text{ and } \exists K \in \cF : U \text{ is open in } K \bs G \right\}.
\]
Let $B = B_1 \cap \cdots \cap B_n$ be in the base generated by $\cB$.
Since $\mu$ is injective, note
\[
\mu[B \x O] ~=~ \mu[B_1 \x O \cap \cdots \cap B_n \x O] ~=~ \mu[B_1 \x O] \cap \cdots \cap \mu[B_n \x O]
\]
is open in $T(i,H)$.
Thus $\mu[O' \x O]$ is open in $T(i,H)$ for any open set $O'$ in $S(i,H)$.
Therefore $\mu[W]$ is open for any open set $W$ in product topology of $S(i,H) \x G$.
As $[-]: S(i,H) \x G \longra S(i,H) \x_H G$ is continuous, $\mu$ is open so a $G$-homeomorphism.
Hence each $e \in E_\cF^\kappa G = (c(\cF \bs G))^I - \{0\}^I$ is a member of some $\cF$-tube $T(i,H)$.

Finally, let $e \in \sE_\cF^\kappa G$.
Then $G_e = G_{e_i}$ for some $i \in I$.
We may assume $e_i \neq 0$.
Then $e_i \in H \bs G$ for some $H \in \cF$.
So $e \in T(i,H) \cap \sE_\cF^\kappa G \approx (S(i,H) \cap \sE_\cF^\kappa G) \x_H G$.
\end{proof}

\section{The classifying property: existence}

\begin{thm}\label{thm:classify}
Let $G$ be a topological group.
Let $\cF$ be any set of subgroups of $G$.
Let $\kappa$ be a cardinal.
Let $X$ be a $G$-space isovariantly covered by $\kappa$-many $\cF$-tubes.
Then $X$ is $G$-homeomorphic to the pullback $f^*(\sE_\cF^\kappa G)$ for some map $f: X/G \longra \sB_\cF^\kappa G$.
The conclusion holds with $E_\cF^\kappa G$ if the cover is only approximate and each $G_x$ closed.
\end{thm}

Recall that each isotropy group $G_x$ is closed if $X$ is Hausdorff ($T_2$) \cite[I:3.5]{tomDieck_TG}.
Any space is \textbf{regular} if any neighborhood of a point has a closed subneighborhood.

\begin{proof}
Let $i \in I$.
There is a $G$-homeomorphism $\phi_i: T_i \longra S_i\x_{H_i} G$ with $H_i \in \cF$.
Write $q_i: S_i \x_{H_i} G \longra H_i \bs G$ for the $G$-map $[s,g] \longmapsto H_i g$.
Define a $G$-map
\[
F_i: X \longra c(\cF \bs G) \quad;\quad x \longmapsto
\begin{cases}
(q_i \circ \phi_i)(x) & \text{if } x \in T_i\\
0 & \text{if } x \notin T_i.
\end{cases}
\]
The $G$-map $F: X \longra E_\cF^\kappa G$ whose $i$-th coordinate is $F_i$ induces $f: X/G \longra B_\cF^\kappa G$.
It remains to prove the following canonical surjective $G$-map is an open injection:
\[
\Psi: X \longra f^*(E_\cF^\kappa G) \quad;\quad x \longmapsto (xG, F(x)).
\]

Fix $i \in I$.
Let $O$ be open in $G$, and let $U$ be open in $S_i$.
Consider the open set
\[
V ~:=~ \phi_i\inv[U \x G] G ~\x~ \{ e \in E_\cF^\kappa ~|~ e_i \subset H_i O \}
\]
in $X/G \x E_\cF^\kappa G$, with the quotient map $X \longra X/G; x \longmapsto xG$ open \cite[I:3.1iv]{tomDieck_TG}.
Note $\Psi(\phi_i\inv[U \x H_i O]) = V \cap f^*(E_\cF^\kappa G)$.
Thus $\Psi|T_i$, hence $\Psi$, is an open function.

It remains to show that $\Psi$ is injective.
Suppose $\Psi(x)=\Psi(y)$ for some $x,y \in X$.
Then $xG=yG$, that is, $y=xa$ for some $a \in G$.
First, assume the cover is isovariant.
There exist $i \in I$ and $g \in G$ such that $x \in T_i$ and $G_x = g\inv H_i g$.
Note $F_i(x) = H_i n g$ for some $n \in N_G(H_i)$.
Then $H_i nga = F_i(x) a = F_i(y) = F_i(x) = H_i ng$.
So $a \in G_x$.
Hence $y=x$.
Alternatively, assume the cover is approximate and $G_x$ is closed in~$G$.
Kolmogorov proved topological groups $G$ are regular in the above sense \cite[\S1]{Kolmogorov_regular}.
Assume $a \notin G_x$.
Then $a \notin O$ for some neighborhood $O$ of $G_x$ in $G$.
There are $i \in I$ and $g \in G$ with $x \in T_i$ and $G_x \leqslant g\inv H_i g \subset O$.
Again note $F_i(x) = H_i ng$ and then $g a g\inv \in H_i$.
So now $a \in O$, a contradiction.
Hence $a \in G_x$.
Therefore $y=x$.
\end{proof}

\begin{rem}\label{rem:final}
Below are earlier classifying spaces that isovariantly map to ours.
Our models $E_\cF^\kappa G$ are $T_0$ but not $T_1$, if $G$ is\footnote{\label{foot:Pontryagin}Kolmogorov ($T_0$) topological groups $G$ are Tikhonov ($T_{3.5}$, completely regular Hausdorff) \cite[\foreignlanguage{russian}{Теорема}~10]{Pontryagin2}.} $T_0$ and each $H \in \cF$ is a closed set in $G$.
A reason to regard higher cardinals $\kappa$ is $G = (S^1)^I$, with the product topology, for any infinite set $I$.
These infinite-dimensional toral groups are connected compact \cite[p830]{Cech} abelian Hausdorff groups and archetypes beyond Lie groups \cite[8.15]{HM}.
The continuum $\kc := 2^{\aleph_0}$ can be $\aleph_\alpha$ for ordinals $\alpha > 0$ in ZFC set theory.
\end{rem}

\subsection{Free actions}

The following was my motivation and stated by G~Segal in russian \cite{Segal}.

\begin{cor}[Segal]\label{cor:Segal}
Let $G$ be any topological group.
Let $\kappa$ be any cardinal.
Let $X$ be a principal $G$-bundle covered by $\kappa$-many local trivializations.
Then $X$ is $G$-homeomorphic to the pullback bundle $f^*(E^\kappa G)$ for some map $f: X/G \longra B^\kappa G$.
\end{cor}

This simplified Gelfand--Fuks' model \cite{GF}\footnote{Also \cite{GF} had only $X/G$ be Hausdorff if $X$ were ``locally $\mathbf{T}$-trivial'' like \cite[D\'{e}finition(a)]{EF}.} where the orbit space $X/G$ is assumed Tikhonov ($T_{3.5}$): a $T_0$ space $Y$ is \textbf{Tikhonov} if points and closed sets are separated by maps $Y \longra~[0,1]$.
An upper bound on $\kappa$ is the \textbf{weight} of $X/G$, the minimum cardinality for a base.
For example, if $X/G$ is second-countable, then $\kappa$ can be taken as the first infinite cardinal $\aleph_0$.

\begin{cor}[Gelfand--Fuks]\label{cor:GF}
Let $G$ be a topological group.
Let $B$ be a Tikhonov space, with weight denoted $wB$.
Let $X$ be a principal $G$-bundle over $B$.
Then $X$ is $G$-homeomorphic to the pullback bundle $f^*(E^{wB} G)$ for some map $f: B \longra B^{wB} G$.
\end{cor}

This served to generalize Dold's pullback \cite{Dold} of Milnor's construction \cite{Milnor_univ2}.
A fast formula for $f$ is given by tomDieck \cite[II]{tomDieck_numerable} and Husem\"{o}ller~\cite[4:12.2]{Husemoller}, who applied Milnor's \emph{countable} partition of unity trick \cite[p25--26]{Milnor_notes} \cite[5.9]{MilnorStasheff}.

\begin{cor}[Milnor--Dold]\label{cor:MilnorDold}
Let $G$ be a topological group.
Let $B$ be a paracompact Hausdorff space.
Let $X$ be a principal $G$-bundle over $B \approx X/G$.
Then $X$ is $G$-homeomorphic to the pullback bundle $f^*(E^{\aleph_0} G)$ for some map $f: B \longra B^{\aleph_0} G$.
\end{cor}

In turn, this bests Steenrod \cite[19.6, 19.3]{Steenrod_book}: $O_n \bs O_{n+k}$ models $E^n G$ for an embedding $G \leqslant O_k$.

\begin{cor}[Steenrod]\label{cor:Steenrod}
Let $G$ be a compact Lie group.
Let $B$ be a finite simplicial complex, say of dimension $n$.
Let $X$ be a principal $G$-bundle over $B$.
Then $X$ is $G$-homeomorphic to the pullback bundle $f^*(E^n G)$ for some map $f: B \longra B^n G$.
\end{cor}

The same conclusion holds for $G$ Lie and $B$ paracompact of $\dim \leqslant n$ \cite[5.10]{Khan_Lie}.

\begin{rem}\label{rem:comparison}
We illustrate how Segal's allowance of non-Hausdorff base spaces (\ref{cor:Segal}) is useful in geometric combinatorics beyond Milnor--Dold's model (\ref{cor:MilnorDold}).
Let $G$ be a discrete finite group.
Recall that a $G$-CW complex is \textbf{regular} if the attaching map of each cell is a homeomorphism.
Consider the problem of enumerating, possibly with repetition of isomorphism classes, all free regular $G$-CW complexes $E$ with orbit space a given connected finite regular CW complex $B$.

By functoriality in Bj\"orner's correspondence \cite[3.1]{Bjoerner}, regular $G$-CW complexes (each admits a canonical simplicial subdivision \cite[III:2.1, III:1.7]{LW_book}) correspond (via the \emph{face poset} consisting of the closed cells under inclusion) to so-called CW posets \cite[2.1]{Bjoerner} with $G$-action; here ``poset'' abbreviates ``partially ordered set.''
Given any CW complex $K$, write $\Delta(K)$ for its face poset and consider the \textbf{Aleksandrov-discrete space} $\| \Delta(K) \|$: the $T_0$ (which is $T_1$ iff $\Delta(K)$ has no comparable elements iff $\dim(K)=0$) topological space on the underlying set of $\Delta(K)$ with the open sets being the \textbf{upper sets} $U$ (that is, $y \supseteq x \in U$ implies $y \in U$) \cite{Aleksandrov}.
Then, by Segal's model (\ref{cor:Segal}), enumeration of all free regular $G$-CW complexes $E$ with $E/G = B$, possibly repeating isomorphism classes, is all of the maps $A := \| \Delta(B) \| \longra B^{wA} G$ between these finite $T_0$ spaces.

On the other hand, if one uses the Milnor--Dold model (\ref{cor:MilnorDold}) in conjunction with obstruction theory, all connected free (hence CW) $G$-spaces $X$ with $X/G = B$ would correspond to the classification of regular $G$-fold coverings: the conjugacy classes of normal subgroups of $\pi_1(B)$ with quotient $G$.
In principle, for small $G$ and finitely presented $\pi_1(B)$, the Dietze--Schaps multistep algorithm \cite{DS} applies to this classical enumeration; however one inputs the order of $G$ and then eliminates the spurious quotient groups of the same order.
Alternatively, our new perspective in terms of conversion to finite $T_0$ spaces directly provides a finite search-space implementable on a computer; a second pass eliminates redundant representatives within isomorphism classes.

For nonfree actions of finite groups $G$ with orbit space a finite regular CW complex $B$, using Mostow's slice theorem for these $T_{3.5}$ $G$-spaces to obtain $\cF$-tubes \cite[1.7.19]{Palais_book}, in principle our method (\ref{thm:classify}) of finite $T_0$ spaces works.
On the other hand, the classical approach of Palais--Bredon (\ref{cor:PalaisBredon}, \ref{thm:unique}) involves (stratified) obstruction theory --- a multistep cohomological process.
\end{rem}

\subsection{Unstructured fiber bundles}

Balanced products $X \x_G F$ allow analogies of the above results for fiber bundles with fiber $F$ and structure group $G$ \cite[2.3]{Steenrod_book}.
However, applications do not effortlessly and formally occur to the more primitive notion of a \emph{fiber bundle with fiber $F$} and no given structure group \cite[1.1]{Steenrod_book}.

Nonetheless, for any base and certain fibers, we can combine \cite[5.4--5.5]{Steenrod_book} with \cite[Theorem~4]{Arens_homeo} to associate a principal bundle to unstructured fiber bundles.

\begin{thm}[Steenrod--Arens]\label{thm:SteenrodArens}
Let $F$ be a Hausdorff space that either is compact or is both locally connected and locally compact.
Endow $\Homeo(F)$ with compact-open topology $G$ \cite{Fox}.
Any $F$-fiber bundle over any topological space $B$ is isomorphic to the balanced product $X \x_G F$ for some principal $G$-bundle $X$ over $B$.
\end{thm}

\begin{cor}\label{cor:unstructured_locconFiber}
Let $F$ be compact $T_2$ space or locally connected locally compact $T_2$.
Endow $\Homeo(F)$ with compact-open topology $G$.
Let $B$ be a topological space;~let $\kappa$ be a cardinal.
Any $F$-fiber bundle $p: E \longra B$ covered by $\kappa$-many local trivializations $\{ U_i \x F \}_{i \in I}$ is isomorphic to pullback $f^*(E^\kappa G \x_G F)$ for a map $f: B \longra B^\kappa G$.
\end{cor}

\begin{proof}
This is immediate from Theorem~\ref{thm:SteenrodArens} and Corollary~\ref{cor:Segal}.
\end{proof}

\begin{rem}[Cianci--Ottina]
For $B$ any Aleksandrov space \cite{Aleksandrov} and the above sort of fiber $F$, Corollary~\ref{cor:unstructured_locconFiber} overlaps with existence of a Grothendieck-type classifying space for unstructured $F$-fiber bundles over $B$ found recently in \cite[4.3]{CO}.
\end{rem}

The following result is well-known nowadays but seems to be undocumented.

\begin{cor}[Holm]
Endow $\Homeo(\R^n,0)$ with compact-open topology $\TOP_n$.
Let $B$ be a paracompact Hausdorff space.
Any $\R^n$-microbundle $p: E \longra B$ \cite{Milnor_microbundles} is isomorphic to pullback $f^*(E^{\aleph_0} \TOP_n \x_{\TOP_n} \R^n)$ for a map $f: B \longra B^{\aleph_0} \TOP_n$.
\end{cor}

\begin{proof}
Holm shows $\R^n$-microbundles over $B$ are $(\R^n,0)$-fiber bundles~\cite[3.3]{Holm}.
As $\R^n$ is locally connected locally compact, use Theorem~\ref{thm:SteenrodArens} and Corollary~\ref{cor:MilnorDold}.
\end{proof}

Analyzing his own main proof ($Y=[0,1]$), Crowell noticed this fact~\cite[\S4]{Crowell}.

\begin{thm}[Crowell]\label{thm:Crowell}
Let $X$ be any locally compact Hausdorff space.
Endow $\Homeo(X)$ with Arens' $g$-topology \cite{Arens_homeo}.
Let $Y$ be any locally connected space.
Any continuous function $h: X \x Y \longra X$ with each $h_y: X \longra X ~;~ x \longmapsto h(x,y)$ a homeomorphism has its adjoint $h^*: Y \longra \Homeo(X) ~;~ y \longmapsto h_y$ being continuous.
\end{thm}

Now, we allow the fibers of Corollary~\ref{cor:unstructured_locconFiber} to include non-locally connected examples, such as the $p$-adic rationals $F=\Q_p$, by transferring the condition to the~base.

\begin{cor}\label{cor:unstructured_locconBase}
Let $F$ be any locally compact Hausdorff space.
Endow $\Homeo(F)$ with Arens' $g$-topology $G$.
Let $B$ be any locally connected topological space; let $\kappa$ be a cardinal.
Any $F$-fiber bundle $p: E \longra B$ covered by $\kappa$-many local trivializations $\{ U_i \x F \}_{i \in I}$ is isomorphic to the pullback $f^*(E^\kappa G \x_G F)$ for a map $f: B \longra B^\kappa G$.
\end{cor}

\begin{proof}
By \cite[Theorem~3]{Arens_homeo}, $G$ is a topological group with continuous evaluation function $G \x F \longra F ~;~ (g,f) \longmapsto g(f)$, and it is the coarsest for which these hold.
For any transition $(U_i \cap U_j) \x F \longra (U_i \cap U_j) \x F \xrightarrow{\pr_F} F$ of local trivializations, since $U_i \cap U_j \subset B$ is locally connected, by Theorem~\ref{thm:Crowell}, its adjoint $U_i \cap U_j \longra G$ is continuous.
So the $F$-fiber bundle $p: E \longra B$ has structure group $G$~\cite[2.3]{Steenrod_book} and is isomorphic to $X \x_G F$ for an associated $G$-principal bundle $X \longra B$~\cite[8.1]{Steenrod_book}.
Then $X$ is $G$-homeomorphic over $B$ to a pullback of $E^\kappa G$, by Corollary~\ref{cor:Segal}.
\end{proof}

Again, notice that an upper bound on the cardinal $\kappa$ is the weight $wB$ of the base space $B$.

\subsection{Nonfree actions}

Bredon \cite[II:9.7i]{Bredon_TG} reworked Palais \cite[2.6.2]{Palais_book}, who had $X$ separable and locally compact.
In the conclusion \cite[4.5]{Palais}, Palais asserts that his classification also holds for any Lie group $G$ if the action is Palais-proper; this assertion is enacted and extended further in our Theorem~\ref{thm:unique}.

\begin{cor}[Palais--Bredon]\label{cor:PalaisBredon}
Let $G$ be a compact Lie group.
Let $\cF$ be a finite set of subgroups of $G$ with no conjugate elements.
Let $X$ be a metrizable $G$-space with all orbit types represented in $\cF$.
Then $X$ is $G$-homeomorphic to the pullback $f^*(\sE_\cF^{\aleph_0} G)$ for some map $f: X/G \longra \sB_\cF^{\aleph_0} G$.
\end{cor}

Recall, when $G$ is a compact Lie group, the Peter--Weyl theorem implies that there are only countably many conjugacy classes of compact subgroups \cite[1.7.27]{Palais_book}.
Next, Ageev had a similar construction to ours in the realm of metric spaces \cite[3.2]{Ageev4}.

\begin{cor}[Ageev]
Let $G$ be a compact Lie group.
Write $\cpt$ for the set of compact subgroups of $G$.
Let $X$ be a metrizable $G$-space (hence $X/G$ is metrizable).
Then $X$ is $G$-homeomorphic to pullback $f^*(E_\cpt^{\aleph_0} G)$ for a map $f: X/G \longra B_\cpt^{\aleph_0} G$.
\end{cor}

Now we generalize this further, from having $X/G$ be metrizable to only Tikhonov.

\begin{cor}\label{cor:classify}
Let $G$ be any Lie group.
Write $\cpt$ for the set of compact subgroups of $G$.
Let $X$ be a Tikhonov space, equipped with a Palais-proper $G$-action.
By Palais' slice theorem \cite{Palais}, $X$ is isovariantly covered by $\cpt$-tubes, say $\kappa$-many.
Then $X$ is $G$-homeomorphic to the pullback $f^*(E_\cpt^\kappa G)$ for some map $f: X/G \longra B_\cpt^\kappa G$.
\end{cor}

In the next definition, the trivial group is in $\cpt$ but not in $\lrg$ for the $p$-adic integers $G = \Z_p$.

\begin{defn}[Antonyan]\label{defn:large}
Let $G$ be a locally compact Hausdorff group.
A subgroup $H$ of $G$ is \textbf{large} if the homogeneous space $G/H$ is a topological manifold.
Write $\lrg \subseteq \cpt$ (equality if $G$ Lie) for the subset of large compact subgroups of $G$.
\end{defn}

The approximate-slice theorem of Abels--Biller--Antonyan \cite[3.6]{Antonyan2} is used instead of Palais' slice theorem to prove the following generalization of Corollary~\ref{cor:classify}.
Our conclusion follows from theirs: their $G$-map $F$ is an embedding \cite[4.4(1)]{AAV}.
We conclude the same; our $F$ also separates points from closed sets \cite[2.3.20]{Engelking2}.

\begin{cor}[Antonyan--Antonyan--Valera-Velasco]
Let $G$ be a locally compact Hausdorff group.
Let $X$ be a Tikhonov space with a Palais-proper $G$-action.
Then $X$ is $G$-homeomorphic to the pullback $f^*(E_\lrg^{wX} G)$ for some map $f: X/G \longra B_\lrg^{wX} G$.
\end{cor}

\begin{proof}
By Theorem~\ref{thm:classify} it suffices to approximately cover $X$ by $wX$-many $\lrg$-tubes.

Let $x \in X$; let $O$ be an open neighborhood of $G_x$ in $G$.
Since the $G$-action on $X$ is Bourbaki-proper \cite[1.4, 1.6c]{Biller_proper}, the $G$-map $G \longra xG ; g \longmapsto xg$ is open \cite[I:3.19iii]{tomDieck_TG}.
Then $xO = U \cap xG$ for an open set $U$ in $X$.
By the approximate-slice theorem \cite[3.6]{Antonyan2}, there exists a \emph{large} compact subgroup $H \geqslant G_x$ of $G$ with $xH \subset U$.
Then $xH \subset xO$.
Hence $H \subset O$.
Moreover, that theorem gives an open $G$-neighborhood $T(x,O)$ of $x$ in $X$ that is $G$-homeomorphic to $S \x_H G$ for some $H$-space $S$ \cite[3.5]{Abels}.
Finally, this approximate cover $\{T(x,O)\}$ of $X$ by $\lrg$-tubes, by the Axiom of Choice, has a subcover\footnote{Note Lindel\"{o}f's lemma is $wX=\aleph_0$: any open cover  has a countable subcover \cite[II: $\R^n$]{Lindelof}.} of cardinality $\leqslant wX$ \cite[1.1.14]{Engelking2}.
\end{proof}

\subsection{Generalized equivariant principal bundles}

\begin{thm}\label{thm:equiv_princ_bund}
Let $\G$ be a topological group; let $\Pi$ be any normal subgroup of $\G$.
Let $\cF$ be a set of subgroups of $\G$ such that $H \cap \Pi = 1$ for each $H \in \cF$.
Let $\kappa$ be a cardinal.
Let $X$ be a $\G$-space isovariantly covered by $\kappa$-many $\cF$-tubes.
Then $X$ is $\G$-homeomorphic to the pullback $c^*(\sE_\cF^\kappa \G)$ for some $\G/\Pi$-map $c: X / \Pi \longra \sE_\cF^\kappa \G / \Pi$.
\end{thm}

Observe $X \longra X/\Pi$ is a principal $\Pi$-bundle, as the restriction of $\cF$ to $\Pi$ is $\{1\}$.

\begin{proof}
Define $c$ using the $F$ and $f$ for $\G$ of Proof~\ref{thm:classify} in the commutative diagram
\begin{equation}\label{eqn:pb1}
\begin{tikzcd}
X \arrow{r} \arrow{d}{F} & X/\Pi \arrow{r} \arrow{d}{c} & X/\G \arrow{d}{f}\\
\sE_\cF^\kappa\G \arrow{r} & \sE_\cF^\kappa\G/\Pi \arrow{r} & \sB_\cF^\kappa\G.
\end{tikzcd}
\end{equation}
Then $X$ is $\G$-homeomorphic to the pullback $f^*(\sE_\cF^\kappa\G)$, by Theorem~\ref{thm:classify}.
Consider the set $\cF_\Pi := \left\{ H_\Pi := H\Pi/\Pi ~|~ H \in \cF \right\}$ of subgroups of the topological group $G := \G/\Pi$.

Note $G_{x\Pi} = (\G_x)_\Pi$.
For each $H \in \cF$, the quotient map $H \longra H_\Pi$ is an isomorphism, since $H \cap \Pi = 1$.
Then any $H$-space $S$ is an $H_\Pi$-space.
So $(S \x_H \G) / \Pi \iso S \x_{H_\Pi} G$ as $G$-spaces.
Thus, by Proof~\ref{prop:cover_classifying}, the $G$-space $\sE_\cF^\kappa\G / \Pi$ is isovariantly covered by $\kappa\cdot\card(\cF)$-many $\cF$-tubes.
Theorem~\ref{thm:classify} gives the commutative diagram
\begin{equation}\label{eqn:pb2}
\begin{tikzcd}
X/\Pi \arrow{r}{c} \arrow{d} & \sE_\cF^\kappa\G / \Pi \arrow{r} \arrow{d} & \sE_\cF^{\kappa\cdot\card(\cF)} G \arrow{d}\\
X/\G \arrow{r}{f} & \sB_\cF^\kappa \G \arrow{r} & \sB_\cF^{\kappa\cdot\card(\cF)} G
\end{tikzcd}
\end{equation}
where the right square and the rectangle are pullbacks.
By the so-called \emph{pasting law} \cite[III:4.8b]{MacLane_book}, the left square of \eqref{eqn:pb2} is a pullback.
Again, since the right square and the rectangle of \eqref{eqn:pb1} are pullbacks, the left square of \eqref{eqn:pb1} is a pullback.
\end{proof}

May--Elmendorf's model~\cite[p278]{Elmendorf} led to this existence part of \cite[Theorem~9]{LM}.
The pullback property \cite[p269]{Lashof_equivbundles} is established in \cite[Theorem~7]{LR}.

\begin{cor}[Lashof--May--Rothenberg]
Let $\Pi$ be a closed normal subgroup of a compact Lie group $\G$.
Write $\cF(\Pi;\G)$ for the set of closed subgroups $H$ of $\G$ with $H \cap \Pi = 1$.
Let $X \longra X/\Pi$ be a numerable $(\Pi;\G)$-bundle \cite{LM} with $X$ Tikhonov.
Suppose that $X$ has the $\G$-homotopy type of a $\G$-CW complex.\footnote{A $(\Pi;\G$)-bundle $X \longra X/\Pi$ is numerable and $X \in T_{3.5}$ if $X$ \emph{is} a $\G$-CW complex \cite[4,5]{LM}.}
Then the $\G$-space $X$ is $\G$-homeomorphic to $c^*(\sE_{\cF(\Pi;\G)}^{\aleph_0} \G)$ for some $\G/\Pi$-map $c: X / \Pi \longra \sE_{\cF(\Pi;\G)}^{\aleph_0} \G / \Pi$.
\end{cor}

\begin{proof}
An isovariant cover exists by Palais~\cite{Palais}; $\kappa = \aleph_0$ by tomDieck~\cite{tomDieck_numerable}.
\end{proof}

A noncompact result exists for twisted equivariant principal bundles \cite[11.4]{LU}.

\begin{cor}[L\"uck--Uribe]\label{cor:LU}
Let $\Pi, G$ be compactly generated Hausdorff groups.
Let $\tau: G \longra \Aut(\Pi)$ be a homomorphism with adjoint $G \x \Pi \longra \Pi$ continuous.
Write $\G := k(\Pi \rtimes_\tau G)$ \cite[III:2.28]{Bourbaki} \cite{Steenrod_convenient}.
Let $X$ be a $\G$-CW complex with $X \longra X/\Pi$ a $(G,\tau)$-equivariant principal $\Pi$-bundle.
Suppose $\cR$ is a family of local representations for $(G,\tau,\Pi)$ \cite[3.3]{LU} that satisfies the (H)-condition \cite[6.1]{LU}.
Write $\cF(\cR)$ for its associated family of closed subgroups of $\G$ \cite[3.5]{LU}.
Then $X$ is $\G$-homeomorphic to the pullback $c^*(\sE_{\cF(\cR)}^{\aleph_0} \G)$ for a $G$-map $c: X / \Pi \longra \sE_{\cF(\cR)}^{\aleph_0} \G / \Pi$.
\end{cor}

\begin{rem}[Guillou--May--Merling]
Let $\Pi$ be either a discrete group or a compact Lie group.
Let $G$ be a discrete group.
Let $\tau: G \longra \Aut(\Pi)$ be any homomorphism.
Write $\G := \Pi \rtimes_\tau G$.
Corollary~\ref{cor:LU} applies with $\cF(\cR) = \cF(\Pi;\G)$.
The model in \cite[0.4]{GMM} is more rigid, as it descends from a categorical framework.
\end{rem}

\section{The classifying property: uniqueness, I}

\subsection{Topological groups $G$ and arbitrary $G$-spaces}

As a prelude to the next subsection, we discuss coarse cones and coarse joins (survey in \cite{FG}).

\begin{defn}\label{defn:coarse_cone}
Let $(A,\cT)$ be a topological space.
Consider the half-smash set
\[
A_+ \wedge [0,1] ~:=~ A \x [0,1] ~/~ (\forall a, a' \in A : (a,0) \sim (a',0)).
\]
The \textbf{coarse cone} $C(A,\cT)$ is this set equipped with the coarsest topology for which the functions $A_+ \wedge [0,1] \longra [0,1] ~;~ [a,t] \longmapsto t$ and $A \x (0,1] \longra A ~;~ [a,t] \longmapsto a$ are continuous.
The \textbf{fine cone} $\sC(A,\cT)$ is that set equipped with the finest topology for which the function $A \x [0,1] \longra A_+ \wedge [0,1] ~;~ (a,t) \longmapsto [a,t]$ is continuous.
\end{defn}

\begin{rem}\label{rem:cone_metric}
For any space $X$, the identity function $\sC X \longra C X$ is continuous.
It is a homeomorphism if $X$ is compact, by the tube lemma.
It is not so for $X=\R$, since $\{ [x,t] \in \R_+ \wedge [0,1] ~|~ t < 1/(1+x^2) \}$ is a neighborhood of the cone point in the fine topology but not the coarse one.
If $d$ is a metric on $X$ then $C X$ has metric
\[
Cd([x,s],[y,t]) ~:=~ |s-t| + \min\{s,t\} \cdot d(x,y).
\]
Note a function $f=(f_0,f_1): A \longra CX$ is continuous if and only if its coordinates $f_1: A \longra [0,1]$ and $f_0: f_1\inv(0,1] \longra X$ are continuous.
A function $\sC X \longra Z$ is continuous if and only if the composition with $X \x [0,1] \longra \sC X$ is continuous.
\end{rem}

\begin{rem}
Historically, these two notions of cone were implicit in the topological study of simplicial complexes $K$.
If $K$ is given the CW topology \cite[p316]{Whitehead_simplicial}, then $\sC K$ is canonically a CW complex.
If $K$ is given the euclidean-metric topology \cite[I.1:4.12]{Lefschetz_TT}, then $C K$ is induced by the euclidean metric\footnote{Given the vertex set $S$ of the abstract simplicial complex \cite[\S IV.1:1]{AH} underlying $K$, there is a geometric realization in terms of basis vectors in the coproduct $\R^{\oplus S}$ equipped with the 2-norm.}.
Recall the CW topology on $K$ is finer than the metric one and is it if and only if $K$ is locally finite.
\end{rem}

If $X$ is a $G$-space, then $C X$ and $\sC X$ have $G$-actions defined by $[x,t] \cdot g := [x g, t]$.

\begin{exm}[Gelfand--Fuks]
Their unrestricted join \cite{GF} of any $G$-spaces is
\[
\bigcoast_{i \in I} X_i ~:=~ \left(\prod_{i \in I} \sC X_i\right) - \{0\}^I.
\]
\end{exm}

Correspondingly, we reformulate Milnor's definition \cite{Milnor_univ2} in terms of cones.

\begin{defn}[Milnor]\label{defn:Milnor_join}
The \textbf{coarse join} of any set of topological spaces is
\[
\mathop{\bigcirc}_{i \in I} X_i ~:=~ \left\{ [x,t] \in \prod_{i \in I} CX_i ~\bigg|~ \exists \text{ finite } J \subseteq I : \sum_{j \in J} t_j = 1 \text{ and } \forall i \in I-J : t_i=0 \right\}
\]
endowed with the subspace topology induced from Tikhonov's product topology.
Write $\bE Z := Z^{\circ \aleph_0}$ for any space $Z$, with the diagonal $G$-action if $Z$ has a $G$-action.
\end{defn}

As just above, we update for arbitrary cardinality, in \cite[1.3.6]{Palais_book} \cite[p108]{Bredon_TG}.

\begin{defn}[Palais]\label{defn:Palais_join}
The \textbf{isovariant join} of a set of topological $G$-spaces is
\[
\bigoasterisk_{i \in I} X_i ~:=~ \left\{ [x,t] \in \mathop{\bigcirc}_{i \in I} X_i ~\bigg|~ \exists i \in I : G_{[x,t]} = G_{x_i} \right\}.
\]
\end{defn}

However, if $I$ is infinite, this `hemorrhages' in neighborhoods of $A\cup X_{(H)}$ in $X$ in \cite[Proof~II:9.5]{Bredon_TG}, obstructing renormalization on $X-A$ to finite support.
This na\"ivet\'e is untenable for Proof~\ref{thm:extensor}; we introduce our own coarse join in the Gelfand--Fuks style, which for finite $I$ is $G$-homeomorphic to Palais' join via normalization.

\begin{defn}\label{defn:Khan_join}
The \textbf{unrestricted isovariant join} of topological $G$-spaces is
\[
\bigbcoast_{i \in I} X_i ~:=~ \left\{ [x,t] \in \prod_{i \in I} CX_i - \{0\}^I ~\bigg|~ \exists i \in I : G_{[x,t]} = G_{x_i} \right\}.
\]
\end{defn}

We remind the reader of the following notion of a proper action \cite[1.2.2]{Palais}.
Note that the Palais-proper condition is automatic if $G$ is an arbitrary compact group.

\begin{defn}[Palais]\label{defn:Palais_action}
A topological $G$-space $X$ is \textbf{Palais} if every $x \in X$ has a neighborhood $U$ in $X$ satisfying: each $y \in X$ has a neighborhood $V$ in $X$ so that the \textbf{transporter} $\langle U, V \rangle_G := \{g \in G ~|~ Ug \cap V \neq \emptyset \}$ is precompact (that is, has compact closure).
Observe that if $G$ is discrete, then $\langle U, U \rangle_G$ being precompact (hence finite) means that the action is properly discontinuous.
\end{defn}

In the second half of this subsection, we quickly construct a filtered homotopy.

\begin{defn}
The following topological space we shall call \textbf{bi-Sierpi\'{n}ski space}:
\[
\bI_2 ~:=~ \left( \{-1,0,+1\},~ \{ \emptyset, \{0\}, \{-1, 0\}, \{0,+1\},  \{-1,0,+1\} \} \right).
\]
It is the particular-point topology on three elements where $\bI_1$ is for two \cite[II:8]{SS}.
The inclusion $\bI_1 \longra \bI_2; 0,1 \longmapsto 1,0$ has left-inverse $\bI_2 \longra \bI_1; -1,0,1 \longmapsto 1,1,0$.
Earlier, $\bI_2$ occurs as the upper topology on the poset of the 1-simplex \cite[I:1.4]{AH}.
Notice the continuous surjection $\Delta: [-1,1] \longra \bI_2 ~;~ \pm 1 \longmapsto \pm1, -1<t<1 \longmapsto 0$.
\end{defn}

\begin{lem}\label{lem:homotopy}
Let $G$ be a topological group.
Let $\cF$ be any set of subgroups of $G$.
Suppose $\{T_i \approx G \x_{H_i} S_i\}_{i \in I}$ and $\{T_j \approx G \x_{H_j} S_j\}_{j \in J}$ each $\cF$-isovariantly cover a $G$-space $X$ for sets $I$ and $J$.
There is a $G$-map $\Phi: X \x \bI_2 \longra \sE_\cF^{I \sqcup J} G$ that restricts to the classifying $G$-maps $F_-: X \x \{-1\} \longra \sE_\cF^I G$ and $F_+: X  \x \{+1\} \longra \sE_\cF^J G$.
The same holds for approximate coverings by $\cF$-tubes where $E_\cF^* G$ (\ref{defn:Segal}) replace $\sE_\cF^* G$.
\end{lem}

So the classifying maps $f_\pm: X/G \longra \sB_\cF^{I\sqcup J} G$ are homotopic via $(\id \x \Delta) \circ \Phi/G$.
The proof works more generally for $G$-maps $F_\pm$ not necessarily induced from tubes.

\begin{proof}
Define $\Phi|: X\x\{0\} \longra \sE_\cF^{I \sqcup J} G$ to be the classifying $G$-map of Theorem~\ref{thm:classify} for the combined isovariant cover $\{T_i\}_{i \in I} \sqcup \{T_j\}_{j \in J}$ of the $G$-space $X$ by $\cF$-tubes.
Define $\Phi| X\x\{-1\}$ to be the classifying $G$-map $F_-$ for the isovariant cover $\{T_i\}_{i \in I}$; define $\Phi| X\x\{+1\}$ to be the classifying $G$-map $F_+$ for the isovariant cover $\{T_j\}_{j \in J}$.
Here we use the $G$-embedding $\sE_\cF^I G \sqcup \sE_\cF^J G \longra \sE_\cF^{I \sqcup J} G$ given by extension by zero.

Let $O$ be open in $G$.
Let $i \in I$.
In the product topology, consider the open set
\[
B(O,i) ~:=~ \{ e \in \sE_\cF^{I\sqcup J} G ~|~ e_i \in OH_i/H_i \}.
\]
There is equipped a $G$-homeomorphism $\phi_i: T_i \longra G \x_{H_i} S_i$.
Note the preimage
\[
\Phi\inv(B(O,i)) ~=~ \phi_i\inv(OH_i \x_{H_i} S_i) \x \{-1,0\}
\]
is open in $X \x \bI_2$.
Similarly one defines $B(O,j)$ for any $j \in J$ and verifies a similar equality.
Observe that $\{B(O,k) ~|~ O \text{ open in } G \text{ and } k \in I \sqcup J \}$ is a subbase for the topology of $\sE_\cF^{I \sqcup J} G$.
Therefore the $G$-function $\Phi$ is continuous.
\end{proof}

\begin{thm}\label{thm:final}
Let $G$ be a topological group.
Let $\cF$ be any set of subgroups of $G$.
Let $\kappa$ be an infinite cardinal.
Then $E_\cF^\kappa G$ is a final object in the category of all topological $G$-spaces covered by $\kappa$-many $\cF$-tubes and $G$-homotopy classes of $G$-maps.
\end{thm}

For paracompact Hausdorff orbit space, one assumes $\kappa  = \aleph_0$ by Proposition~\ref{prop:countable}.

\begin{proof}
Existence is in Proof~\ref{thm:classify}, without ``approximate'' for the extra ``pullback.''
Uniqueness up to $G$-homotopy is Lemma~\ref{lem:homotopy} via $(\id \x \Delta) \circ \Phi$ and $\kappa + \kappa = \kappa$.
\end{proof}

Here is an important case \cite[I:6.6]{tomDieck_TG}, upon which the Baum--Connes conjecture is formulated.
Asserted in \cite[Proof~A:1]{BCH}, one must replace Husem\"oller's case of $\cF = \{1\}$ with L\"uck's observation \cite[2.5i]{Lueck_classify} that the former case works for noncompact $G$.
(Earlier, tomDieck had a narrower case \cite{tomDieck_orbittypes} derived from \cite{tomDieck_numerable}.)

\begin{cor}[tomDieck]
Let $G$ be a locally compact Hausdorff group.
Let $\cF$ be a set of closed subgroups of $G$ preserved under finite intersections and under conjugacy.
The coarse join $\bE(\cF \bs G) = (\cF \bs G)^{\circ \aleph_0}$  is a final object in the full subcategory of numerable $G$-spaces. (Recall the isovariant $G$-map $\bE(\cF \bs G) \longra E_\cF^{\aleph_0} G$ of \ref{rem:final}.)
\end{cor}

The first case is recorded independently in \cite[4:12.4]{Husemoller} \cite[\S3]{tomDieck_numerable} after Milnor.

\begin{cor}[Husem\"oller]
Let $G$ be a topological group.
Then $\bE G$ is a final object in the category of numerable $G$-spaces and $G$-homotopy classes of $G$-maps.
\end{cor}

To state a stronger uniqueness, we require the notion of a stratified homotopy.
The following definition we amplify to preorders from partial-orders~\cite[2.6]{Hughes}; we shall need it in such generality and cannot assume closedness if $X$ is non-Hausdorff.

\begin{defn}[Hughes]\label{defn:stratified}
Let $\cP$ be a set with a preorder\footnote{Recall a \textbf{preorder} is a partial-order without antisymmetry: $a \preccurlyeq a$ holds and $a \preccurlyeq c$ if $a \preccurlyeq b \preccurlyeq c$.} $\preccurlyeq$.
A topological space $Y$ shall be \textbf{$(\cP,\preccurlyeq)$-filtered} if it is equipped with a set $\{Y^a\}_{a \in \cP}$ of subspaces where $Y = \bigcup_{a \in \cP} Y^a$ and $b \prec a$ implies $Y^b \subseteq Y^a$.
A continuous function $f: Y \longra Z$ of $(\cP,\preccurlyeq)$-filtered spaces is \textbf{$(\cP,\preccurlyeq)$-filtered} if $f(Y^a) \subseteq Z^a$ for each $a \in \cP$.
In particular, a map $f: Y \longra Z$ shall be \textbf{$(\cP,\preccurlyeq)$-stratified}\footnote{If $(\cP,\preccurlyeq)$ has upper topology \cite[before~II]{Aleksandrov}, $\preccurlyeq$ satisfies antisymmetry ($a=b$ if $a \preccurlyeq b \preccurlyeq a$), and each set $Y^a$ is closed, then $\{Y_a\}_{a \in \cP}$ is a $(\cP,\preccurlyeq)$-stratification in the sense of Lurie~\cite[A.5.1]{Lurie_HigherAlgebra}. If further the partition $\{Y_a\}_{a \in \cP}$ is locally finite, then it is a $(\cP,\preccurlyeq)$-decomposition in the sense of Goresky--MacPherson~\cite[1.1]{GM_book}. If the partially ordered set is finite and $Y^a$ is closed cofibrant in $Y^b$ if $a \prec b$, then $\{Y^a\}_{a \in \cP}$ is a $(\cP,\preccurlyeq)$-filtration of $Y$ in the sense of Weinberger~\cite[p115]{Weinberger_book}.} if $f(Y_a) \subseteq Z_a$ for each $a \in \cP$, where
\[
Y_a ~:=~ Y^a ~-~ \bigcup_{b \prec a} Y^b.
\]
The source $Y \x [-1,1]$ of a homotopy has stratification $(Y \x [-1,1])_a = Y_a \x [-1,1]$.
\end{defn}

\begin{exm}\label{exm:stratified}
Let $G$ be a topological group.
Let $\cF$ be a set of subgroups of $G$.
Write $(\cF)$ for the set of $G$-conjugacy classes of elements of $\cF$.
Define a preorder $\geqslant$ on $(\cF$) by: $(H) \geqslant (K)$ if $H$ contains a $G$-conjugate of $K$ (the reverse of \cite[3.5]{Khan_Lie}).
Let $X$ be a $G$-space with orbit types in $\cF$.
The \textbf{orbit-type filtration} of $X/G$ is
\[
(X/G)^{(H)} ~:=~ X^{(H)}/G ~=~ \left\{ xG \in X/G ~|~ \exists g \in G : H \subseteq G_{xg} \right\}.
\]
The \textbf{orbit-type stratification} of the orbit space $X/G$ is
\[
(X/G)_{(H)} ~:=~ X_{(H)}/G ~=~ \left\{ xG \in X/G ~|~ \exists g \in G : H = G_{xg} \right\}.
\]
By isovariance, the map $f$ (\ref{thm:classify}) is stratified and homotopy $\Phi/G$ (\ref{lem:homotopy}) is filtered\footnote{To see $\Phi/G: X/G \x \bI_2 \longra \sB_\cF^{I \sqcup J} G$ need not be stratified, take $K$ cohopfian in $X=K\bs G$ and $\kappa = 1, T_1 = X = T_2$ and $\phi_1(Kg) = Kg, \phi_2(Kg) = Kag$ with $a \notin N_G(K)$. Note $G_{H(K,-1)} = G_{(K,0)} = K$ and $G_{H(K,1)} = G_{(0,Ka)} = a\inv K a$, but note $G_{H(K,0)} = G_{(K,Ka)} = K \cap a\inv K a \neq K.$}.
\end{exm}

\begin{rem}\label{rem:BS}
In the preceding example, if $G$ is a Lie group and $\cF \subseteq \cpt(G)$, then it follows from Cartan's closed-subgroup theorem that $\geqslant$ is moreover a partial-order.
However, even for the solvable Baumslag--Solitar group
\[
G ~=~ BS(1,4) ~=~ \gens{ x,y ~|~ yxy\inv = x^4 }
~\cong~ \Z[{\textstyle\frac{1}{4}}] \rtimes_4 \Z,
\]
which is a 0-dimensional Lie group with the discrete topology, antisymmetry of $\geqslant$ fails for $\cF = \{ \gens{x}, \gens{x^2} \}$.
Therefore, for general $G$ and $\cF$, our Definition~\ref{defn:stratified} of filtered spaces is stated in terms of preorders, not the more familiar partial-orders.
\end{rem}

\subsection{Arbitrary Lie groups $G$ and isometric $G$-actions}

At first, the filtered homotopy $\Phi/G$ (\ref{lem:homotopy}) had a stratified strengthening \cite[\S2.7]{Palais_book} \cite[II:9.7]{Bredon_TG}.

\begin{thm}[Palais--Bredon]\label{thm:unique_PB}
Let $G$ be a compact Lie group.
Let $\cF \subseteq \cpt(G)$ be finite with no conjugate elements.
Consider Palais' join of Milnor's join (\ref{defn:Milnor_join}):
\[
\bE_\cF^n G ~:=~ \bigoasterisk_{H \in \cF} (H \bs G)^{\circ n}
\]
for some $n \in \N$.
Let $B$ be an $(\cF)$-filtered metrizable space of covering dimension~$< n$.
Suppose $f, g: B \longra \bB_\cF^n G := \bE_\cF^n G / G$ are stratified maps.
If $f^*(\bE_\cF^n G)$ and $g^*(\bE_\cF^n G)$ are $G$-homeomorphic over the identity $\id_B$, then there exists a stratified homotopy from $f$ to $g$.
\end{thm}

We generalize $G, \cF, n$ following their strategy, but we implement it differently.

\begin{thm}\label{thm:homotopy}
Let $G$ be an arbitrary Lie group.
Let $\cF \subseteq \cpt(G)$ with no conjugate elements.
Consider our unrestricted isovariant join (\ref{defn:Khan_join}) of copies of Milnor's infinite join:
\[
\bE_\cF G ~:=~ \bigbcoast_{H \in \cF} \bE(H \bs G).
\]
Let $B$ be an $(\cF)$-filtered metrizable space.
Suppose $f, g: B \longra \bB_\cF G := \bE_\cF G / G$ are stratified maps (\ref{defn:stratified}).
If $f^*(\bE_\cF G)$ and $g^*(\bE_\cF G)$ are $G$-homeomorphic over the identity $\id_B$, then there exists a stratified homotopy from $f$ to $g$.
\end{thm}

The proof appears after some lemmas in the spirit of the Palais--Bredon strategy.

The first lemma generalizes \cite[II:9.2]{Bredon_TG} without using transfinite induction.
Therein, the members of $\cC$ had dimension $\leqslant n$ and $F$ was a compact $(n-1)$-connected polytope \cite[II:9.1]{Bredon_TG}.
Our $\cC$ shall be the class $\cM$ of metrizable spaces.
Recall that $Z$ is an \textbf{absolute extensor} for $\cC$, written $Z \in \text{AE}(\cC)$, means that for any $X \in \cC$ and closed subset $A \subset X$, any map $A \longra Z$ has an extension $X \longra Z$.

\begin{lem}\label{lem:extendsection_fiberbundle}
Let $\cC$ be a subclass of the class $\cP$ of paracompact Hausdorff spaces, such that any closed subset of any member of $\cC$ is a member of $\cC$.
Let $p: E \longra X$ be an $F$-fiber bundle with any structure group \cite[2.3]{Steenrod_book} and $X \in \cC$ and $F \in \text{AE}(\cC)$.
For any closed subset $A$ of $X$, any section $A \longra E$ extends to a section $X \longra E$.
\end{lem}

A structure group does not occur in \cite[II:9.2]{Bredon_TG} but does in his applications.

\begin{proof}
Since $X \in \cP$, there is a $p$-trivializing locally finite open cover of $X$.
The associated principal bundle has a trivializing locally finite open cover $\{U_i\}_{i=1}^\infty$ that is \emph{countable} \cite[Hilfsatz~2]{tomDieck_numerable}, which works also for the $F$-fiber bundle $p$.
There is a closed refinement $\{C_i \subset U_i\}_{i=1}^\infty$ \cite[2]{Dowker4}, which is locally finite and $p$-trivializing.

Write $A_0 := A$ and $A_{n+1} := A_n \cup C_{n+1}$.
Inductively assume a section $A_n \longra E$ exists extending $A \longra E$ for some $n > 0$.
Since $p\inv(C_{n+1}) \approx C_{n+1} \x F$, sections $C_{n+1} \longra E$ correspond bijectively to maps $C_{n+1} \longra F$.
Then the section $A_n \cap C_{n+1} \longra E$ corresponds to a map $A_n \cap C_{n+1} \longra F$.
Since $A_n$ is closed in $X$, we have $A_n \cap C_{n+1}$ is closed in $C_{n+1} \in \cC$.
Then there is an extension $C_{n+1} \longra F$.
Equivalently, the section $A_n \cap C_{n+1} \longra E$ extends to a section $C_{n+1} \longra E$.
By the pasting lemma, we obtain a section $A_{n+1} \longra E$.
We are done by induction.
\end{proof}

T\,O~Banakh proved the following observation using direct methods \cite[1.3]{Banakh}.
Indirectly, this already followed from Haver \cite{Haver} with Dold \cite[Proof~8.1]{Dold}.

\begin{lem}[Banakh]\label{lem:Banakh}
Let $W$ be a Lie group.
Then Milnor's join $\bE W \in \text{AE}(\cM)$.
\end{lem}

For any class $\cC$ of topological $G$-spaces, a $G$-space $Z$ is an \textbf{absolute $G$-extensor} for $\cC$, written $Z \in G\text{-AE}(\cC)$, if for any closed $G$-subset $A \subset X \in \cC$, any $G$-map $A \longra Z$ extends to a $G$-map $X \longra A$.
Write $G\text{-}\cM$ for the class of Palais (\ref{defn:Palais_action}) $G$-metrizable spaces.
Here, \textbf{$G$-metrizable} means there is a $G$-invariant metric.
Furthermore, a $G$-space $Z$ is an \textbf{absolute neighborhood $G$-extensor} for $\cC$, written $Z \in G\text{-ANE}(\cC)$, if for any closed $G$-subset $A \subset X \in \cC$, any $G$-map $A \longra Z$ admits an extension to a $G$-map $U \longra Z$ for some $G$-neighborhood $U$ of $A$ in $X$.

\begin{lem}\label{lem:notBanakh}
Let $K$ be a compact Lie group.
For any \textbf{$K$-normed linear space} $V$: a vectorspace with linear $K$-action and $K$-invariant norm, $\bE V \in K\text{-ANE}(K\text{-}\cM)$.
\end{lem}

A nonexample is the circle group $K=U_1$ and $V = C_b(\C, \R)$ with the sup-norm, due to failure of continuity of right-action $(f \cdot g)(x) := f(xg)$ \cite[Example~8:1]{Antonyan1}.

\begin{proof}
The \textbf{topological product} $P := (V\oplus\R)^{\aleph_0}$ is the algebraic direct product of $\R$-vectorspaces with product topology.
The topological vectorspace $P$ is metrizable and \emph{locally convex} but not normable.
(Also $P$ is Fr\'echet if and only if $V$ is Banach.)
To see that $P$ is locally convex \cite[\S II:4.1]{Bourbaki3}, recall $P$ has the coarsest topology for which each $n$-th projection $P \longra V\oplus\R$ is continuous.
The norm-topology on $V\oplus\R$ is the coarsest for which the norm and all of its vectorspace translations are continuous.
Thus the topology on $P$ is the coarsest for which each $n$-th seminorm, given by $n$-th projection then $n$-th norm, and all coordinatewise vectorspace translations are continuous \cite[\S II:1.2]{Bourbaki3}.
So $P$ is locally convex \cite[\S II:4.1]{Bourbaki3}.
The proof that $P$ is metrizable formally generalizes that for $V=0$: $\R^{\aleph_0}$ \cite[p547]{Tikhonov}.
Lastly, $P$ is not normable because any basic open neighborhood of $0$ contains a line.

Factoring through metric orbit spaces, by Tietze's extension theorem \cite[Satz~3]{Tietze}, we obtain that $\R \in K\text{-AE}(K\text{-}\cM)$.
Using the coordinate projections, observe that the $K$-action on $P$ is continuous; also $P \in K\text{-AE}(K\text{-}\cM)$ if $V \in K\text{-AE}(K\text{-}\cM)$.
Above, we implicitly used the sup-norm on $V \oplus \R$, namely: $\|(v,r)\| := \max\{\|v\|, |r|\}$.

Consider the bounded level-preserving $K$-injection from the coarse cone (\ref{defn:coarse_cone}):
\[
\iota: CV \longra V \oplus \R ~;~ [x,t] \longmapsto \left( \frac{tx}{1+\|x\|}, t \right).
\]
The restriction $\iota|(CV - \{0\})$ away from the conepoint is an embedding.
Note
\[
\iota\inv\{ (v,r) ~|~ \|(v,r)\| < \eps \} ~=~ \{ [x,t] ~|~ t < \eps \}.
\]
Thus $\iota$ is both continuous and open at the coarse conepoint.
So $\iota$ is a $K$-embedding.
Hence the product function $\iota^{\aleph_0}: (CV)^{\aleph_0} \longra P$ is a $K$-embedding.
Since the image $\iota(CV)$ in $V\oplus\R$ is convex, it follows that $\iota^{\aleph_0}(\bE V) \subset \iota^{\aleph_0}(CV)$ in $P$ is also convex.
Therefore, since $\bE V$ admits a $K$-embedding as a convex $K$-subset of a locally convex vectorspace $P$, and since $K$ is a compact Lie group, by Antonyan's partial generalization \cite{Antonyan5} of Dugundji's extension theorem, $\bE V \in K\text{-ANE}(K\text{-}\cM)$.
\end{proof}

The above variations of extensor, if a member of $\cC$ also, are spaces $Z$ which have the stated extension property specialized to when $A \longra Z$ is the identity map.
They are forms of the \textbf{retract} notion, denoted by the letter R instead of E \cite{Hu}.

The Lie hypothesis of Lemma~\ref{lem:notBanakh} is necessary; if $K$ is a non-Lie metric compact group, there are $K$-normed linear spaces not in $K\text{-ANE}(K\text{-}\cM)$ \cite[Theorem~6]{Antonyan1}.
Lemma~\ref{lem:Banakh} is a case of Banakh's lemma \cite[1.3]{Banakh} stated for all $W \in \text{ANR}(\cM)$.
The latter lemma shall be the $G=1$ case of the following equivariant generalization.

\begin{thm}\label{thm:Banakh_generalized}
Let $G$ be a Lie group.
Milnor's join (\ref{defn:Milnor_join}) defines a class-function
\[
\bE : G\text{-ANR}(G\text{-}\cM) \longra \text{AR}(\cM) \cap G\text{-}\cM \cap G\text{-ANE}(G\text{-}\cM) ~;~ Z \longmapsto \bE Z := Z^{\circ \aleph_0}.
\]
\end{thm}

We mostly repeat the second half of Banakh's proof and introduce Palais actions.
Also we remove his intermediate need for a convex subset and we fill in some details.
I did not fully understand the first half of Banakh's proof, which involved some sort of abstract convexity structure and an appeal to the proof of Dugundji's theorem, so I replaced it with my own Lemma~\ref{lem:notBanakh} which applies Antonyan's rigorous work.

\begin{proof}
Fix $Z \in G\text{-ANR}(G\text{-}\cM)$.
Since $G$ is locally compact and $Z \in G\text{-}\cM$, there exist a $G$-normed linear space $L$, a normed linear space $N$, and a closed $G$-embedding $e: Z \longra (L-0)\x N$ with open $G$-subset $L-0$ Palais \cite[3.10]{AAR}.
Since the $G$-action on $N$ is trivial, $(L-0)\x N$ is Palais (\ref{defn:Palais_action}) hence lies in $G\text{-}\cM$.
Then, since $Z \in G\text{-ANR}(G\text{-}\cM)$, there exists a $G$-retraction $r: O \longra e(Z)$ for some $G$-neighborhood $O$ of $e(Z)$ in $(L-0)\x N \subset L \oplus N$.
Consider the $G$-invariant map
\begin{equation}\label{eqn:eta}
\eta: L \oplus N \longra [0,1] ~;~ x \longmapsto \frac{d(x,C)}{d(x,eZ) + d(x,C)}
\end{equation}
where $C := L\oplus N - O$ and $d(x,S) := \inf_{y \in S} \|x-y\|$.
Note $\eta(C) = \{0\}$ and $\eta(eZ) = \{1\}$.
So $\eta$ for $\cM$ realizes the conclusion of Urysohn's lemma for $T_4$ spaces.

Using $r$ and $\eta$, next we reproduce Banakh's neighborhood retraction $R$ of $\bE (eZ)$ in $\bE(L\oplus N)$, and the map $R$ shall turn out to be $G$-equivariant.
Define a $G$-function
\[
s: \bE(L\oplus N) \longra [0,1] ~;~ [x,t] \longmapsto \sum_{i=0}^\infty \eta(x_i) t_i.
\]
To prove Banakh's assertion that $s$ is continuous, for any $i \in \N$ consider the function $\eta(x_i) t_i: C(L\oplus N) \longra [0,1]$ defined on the coarse cone of a $G$-normed linear space.
It is continuous away from the conepoint, since $\eta$ and multiplication are continuous.
Given $\eps>0$, taking $\delta = \eps$, if $| t_i - 0 | < \delta$ then $| \eta(x_i) t_i - 0 | \leqslant t_i < \eps$, so it is continuous.
Thus, as the $i$-th projection is continuous, for all $n \in \N$, the $n$-th partial sum is too:
\[
s_n : C(L\oplus N)^{\aleph_0} \longra [0,n] ~;~ [x,t] \longmapsto \sum_{i=0}^n \eta(x_i) t_i.
\]
Let $[x,t] \in \bE(L\oplus N)$.
Then $t_i = 0$ for some $n \geqslant 0$ and all $i>n$.
Let $\eps > 0$.
Since $s_n$ and each $[x,t] \longmapsto t_i$ are continuous, then in the product topology (see \ref{defn:Milnor_join}), there exists an open neighborhood $U$ of $[x,t]$ in the subspace $\bE(L\oplus N) \subset C(L\oplus N)^{\aleph_0}$ such that: if $[x',t'] \in U$ then $| s_n[x,t] - s_n[x',t'] | < \eps/2$ and $\sum_{i=0}^n t'_i > 1-\eps/2$; note
\begin{gather*}
s[x,t] - s[x',t'] = s_n[x,t] - s[x',t'] \leqslant s_n[x,t] - s_n[x',t'] < \eps/2 < \eps\\
s[x',t'] - s[x,t] = \sum_{i>n}^\infty \eta(x'_i) t'_i + s_n[x',t'] - s_n[x,t] < \sum_{i>n}^\infty t'_i + \eps/2 < \eps
\end{gather*}
so $| s[x,t] - s[x',t'] | < \eps$.
Thus $s$ is continuous.
Banakh's neighborhood retraction~is
\[
R : s\inv(0,1] \longra \bE (eZ) ~;~ [x,t] \longmapsto \left[ r(x_i), \frac{\eta(x_i) t_i}{s[x,t]} \right].
\]

Let $K \in \cpt(G)$.
By Lemma~\ref{lem:notBanakh}, $\bE(L\oplus N) \in K\text{-ANE}(K\text{-}\cM)$, as $G$ hence $K$ is Lie by Cartan's closed subgroup theorem \cite[27]{CartanE}.
Then $\bE Z \in K\text{-ANE}(K\text{-}\cM)$,~as $\bE Z$ is a neighborhood $K$-retract of $\bE(L\oplus N)$.
Thus $\bE Z \in G\text{-ANE}(G\text{-}\cM)$, by Antonyan's neighborhood version \cite[Thm~5]{Antonyan3} of Abels' induction theorem \cite[4.2]{Abels}.
Also, as $Z \in G\text{-}\cM$, the coarse cone $CZ \in G\text{-}\cM$ by Remark~\ref{rem:cone_metric}.
So the induced metric \cite[20.5]{Munkres} on the countable product $(CZ)^{\aleph_0}$ is $G$-invariant.
Hence $\bE Z \in G\text{-}\cM$.

Finally, we establish the fact that $\bE Z \in AE(\cM)$, independently of \cite[1.3]{Banakh}.
Take $G=1$ in the above arguments and simplify, as follows.
Take $L=0$ and the closed embedding $e: Z \longra N$ with $N$ Arens--Eells' space for $Z \in \cM$ \cite{ArensEells}.
Obtain $R$ as above.
In Lemma~\ref{lem:notBanakh} for $K=1$, replace \cite{Antonyan5} with Dugundji's extension theorem \cite{Dugundji}, to find $\bE N \in \text{AE}(\cM)$.
Then $\bE Z \in \text{ANE}(\cM)$, skipping \cite{CartanE} and \cite{Antonyan3}.
So $\bE Z \in \text{AE}(\cM)$ \cite[III:7.2]{Hu}, since $\bE Z$ is contractible \cite[8.1]{Dold}.
\end{proof}

Next, let $\cK$ be a set of compact subgroups of a locally compact Hausdorff~group~$G$.
Write $(G,\cK)\text{-}\cM$ for the Palais $G$-metrizable spaces with isotropy \emph{conjugate} into $\cK$.

\begin{lem}\label{lem:Gextensor}
Fix $H \in \lrg(G)$ (see \ref{defn:large}) for a locally compact Hausdorff group $G$.
Then the Palais $G$-space $\bE(H \bs G)$ is a member of the class $G\text{-AE}((G,\{H\})\text{-}\cM)$.
\end{lem}

This update of \cite[II:9.3]{Bredon_TG} generalizes half of a recent theorem \cite[4.3]{ZAA}.

\begin{cor}[Zhang--Antonyan--Antonyan]
Let $G$ be a Lie group.
The trivial group $1$ is a large compact subgroup of $G$, hence $\bE G = \bE(1 \bs G) \in G\text{-AE}((G,\{1\})\text{-}\cM)$.
\end{cor}

\begin{rem}\label{rem:large}
Let $G$ be a locally compact group.
Updating his own earlier work, S~Antonyan defines a \emph{closed} subgroup $H$ as \textbf{large in $G$} to be that $G/N$ is a Lie group for some normal subgroup $N \subseteq H$ of $G$ \cite[3.1]{Antonyan2}.
By \cite[3.2]{Antonyan2}, this is equivalent to our Definition~\ref{defn:large}.
When $G$ is separable compact Hausdorff, this equivalence is immediate from \cite[\foreignlanguage{russian}{Теорема}~75]{Pontryagin2} or \cite[Theorem~6.3:1]{MZ_book}, since the kernel of the $G$-action on $G/H$ is the normal subgroup $N = \bigcap_{g \in G} g H g\inv$.

As the large subgroup $H$ is closed in the Hausdorff group $G$, so is its normalizer $N_G(H)$.
The closed subgroup $N_G(H)/N$ of the Lie group $G/N$ is Lie \cite[27]{CartanE}.
Since $H/N$ is closed in $N_G(H)/N$, then $W_G(H) := N_G(H)/H$ is Lie \cite[I:5.3]{tomDieck_TG}.
\end{rem}

Recall the \textbf{$H$-skeleton} (or $H$-fixed set) and \textbf{$H$-stratum} are the $W_G(H)$-spaces
\[
X^H ~:=~ \{x \in X ~|~ H \leqslant G_x \} \quad\text{and}\quad X_H ~:=~ \{x \in X ~|~ H = G_x \}.
\]

Our approach shall \emph{avoid} the related and noteworthy criterion of James--Segal: for any compact Lie group $K$, a member of $K\text{-ANE}(K\text{-}\cP)$ belongs to $K\text{-AE}(K\text{-}\cP)$ if and only if each $H$-skeleton belongs to $\text{AE}(\cP)$ for all closed $H \leqslant K$ \cite[4.1]{JS}.

\begin{proof}[Proof of Lemma~\ref{lem:Gextensor}]
For easier reading, shorten $W := W_G(H)$ and $E := \bE(H\bs G)$.
Observe that $E$ has all orbit types\footnote{\label{ft:noantisymmetry}Like Remark~\ref{rem:BS}, there are large compact subgroups $H$ of locally compact Hausdorff groups $G$ and $g \in G$ with $g H g\inv \subsetneq H$, e.g.~the infinite-dimensional toral group $H = (U_1)^{\aleph_0}$ in $G = H \rtimes \Z$.} $\sqsupseteq (H)$ and that $E^H = \bE (H \bs N_G(H)) = \bE W$.

Recall from \cite[2.6]{Khan_Lie} the space $M_G(X,Y) := \{(x,y) \in X \times Y ~|~ G_x \leqslant G_y \}/G$.
Let $X$ be a $G$-metrizable space with Palais $G$-action of single orbit type $(H)$.
Note
\[
M_G(X,E) ~=~ X_H \x_{N_G(H)} E^H ~=~ X_H \x_W \bE W.
\]
The map $\pi: M_G(X,E) \longra X/G$ becomes $X_H \x_W \bE W \longra X_H/W$, which is an $\bE W$-fiber bundle, as $W$ is Lie (Remark~\ref{rem:large}), by Palais' slice theorem \cite[2.3.1]{Palais}.

Suppose that $A$ is a closed $G$-subset of $X$ and $f: A \longra E$ is a $G$-map.
The $G$-extensions of $f$ to $X$ correspond bijectively \cite[2.6]{Khan_Lie}\footnote{Indeed $E \in \cM \subset T_{3.5}$, since $H \bs G \in \cM$ (\ref{defn:large}) so $C(H \bs G), C(H \bs G)^{\aleph_0} \in \cM$ (\ref{rem:cone_metric}) \cite[4.2.4]{Engelking2}.} to the extensions of the $\pi$-section $\Gamma f: A/G = A_H/W \longra X_H \x_W \bE W$ to $\pi$-sections from $X/G = X_H/W$.
The latter exists by Lemma~\ref{lem:extendsection_fiberbundle}, since $\bE W \in \text{AE}(\cM)$ (\ref{lem:Banakh}) and $\cM \subset \cP$ \cite{Stone1}.
\end{proof}

Extending the above notions, the \textbf{$(H)$-skeleton} and \textbf{$(H)$-stratum} are $G$-spaces
\[
X^{(H)} ~:=~ \{x \in X ~|~ \exists g \in G : H \leqslant G_{xg} \} \enspace\text{and}\enspace X_{(H)} ~:=~ \{x \in X ~|~ \exists g \in G : H = G_{xg} \}.
\]

\begin{lem}\label{lem:closed}
Let $H$ be a compact subgroup of a Lie group $G$.
Let $X$ be a Tikhonov space with Palais $G$-action.
Both $X^{(H)}$ and $X^{(H)}-X_{(H)}$ are closed subsets of $X$.
\end{lem}

So $X_{(H)}$ is locally closed in $X$ with closure $\subseteq X^{(H)}$; see \cite[p68]{Bredon_TG} \cite[I:6.2]{tomDieck_TG}.
Therefore, $X$ satisfies the Frontier Condition over the poset $(\cpt,\geqslant)$ \cite[I:1.1]{GM_book}.

\begin{proof}
We use the notation of Example~\ref{exm:stratified}.
Let $x \in X - X^{(H)}$.
Then $(H) \nleqslant (G_x)$.
By Palais' slice theorem \cite[2.3.1]{Palais}, there exist a $G$-neighborhood $U$ of $x G$ in $X$ and a $G$-retraction $U \longra x G \approx G_x \bs G$.
If $y \in U$ then $(G_y) \leqslant (G_x)$ so $(H) \nleqslant (G_y)$.
Thus $U \subseteq X - X^{(H)}$.
Therefore $X^{(H)}$ is closed in $X$.

Let $a \in X_{(H)}$.
Then $(H) = (G_a)$.
By Palais' slice theorem, there exist a $G$-neighborhood $O$ of $a G$ in $X$ and $G$-retraction $O \longra a G \approx G/H$.
Since $G$ is Lie, its closed subgroups are cohopfian, so the preorder $\leqslant$ is antisymmetric; see Footnote~\ref{ft:noantisymmetry} and \cite[I:3.7]{tomDieck_TG}.
If $b \in X^{(H)} \cap O$ then $(H) \leqslant (G_b) \leqslant (H)$ so $(H)=(G_b)$.
Thus $X^{(H)} \cap O \subseteq X_{(H)}$.
So $X_{(H)}$ is open in $X^{(H)}$.
Hence $X^{(H)} - X_{(H)}$ is closed in $X$.
\end{proof}

Finally, we update Palais--Bredon's key cone lemma \cite[\S2.7]{Palais_book} \cite[II:9.4]{Bredon_TG}.
They only had considered compact $G$, so they equivalently used the fine cone (\ref{rem:cone_metric}).

\begin{lem}\label{lem:PB_cone}
Fix a compact subgroup $H$ of an arbitrary Lie group $G$.
Let $A$ be a closed $G$-subset of a Palais $G$-metrizable space $X$.
Any $G$-map $\vphi : A \longra C\bE(H\bs G)$ with $0 \notin \vphi(A_{(H)})$ admits a $G$-extension $\Phi: X \longra C\bE(H\bs G)$ such that $0 \notin \Phi(X_{(H)})$.
\end{lem}

\begin{proof}
Since $\vphi$ is equivariant, note $\vphi(A^{(H)}-A_{(H)}) = \{0\}$, the coarse conepoint.
Write $Z := \vphi\inv(0)$ and $E := \bE(H\bs G)$.
Since $A_{(H)} \cap Z = \emptyset$, there are coordinates
\[
\vphi|A_{(H)} = (\vphi_0, \vphi_1): A_{(H)} \longra E \x (0,1].
\]
Since $H \in \lrg(G)$ and since $A_{(H)}$ is closed in $X_{(H)} \in (G,\{H\})\text{-}\cM$, by Lemma~\ref{lem:Gextensor}, $\vphi_0: A_{(H)} \longra E$ extends to a $G$-map $\vphi_0': X_{(H)} \longra E$.
Also, since $A_{(H)}$ is closed in $X_{(H)} \in \cM$, by Tietze's extension theorem \cite[Satz~3]{Tietze}, $\vphi_1: A_{(H)} \longra (0,1]$ extends to a $G$-map $\vphi_1': X_{(H)} \longra (0,1]$.
Since $A_{(H)}$ and $B := X^{(H)}-X_{(H)}$ are disjoint sets (with $B$ closed in $X$ by Lemma~\ref{lem:closed}), similar to \eqref{eqn:eta}, construct a map
\[
\eta: X^{(H)} \longra [0,1] ~;~ x \longmapsto \frac{d(x,B)}{d(x,A_{(H)}) + d(x,B)}
\]
satisfying $\eta(B)=\{0\}$ and $\eta(A_{(H)})=\{1\}$.
Then $\vphi_1$ extends to $\vphi_1' \eta: X^{(H)} \longra [0,1]$ which has $B$ being the preimage of $0$.
So $\vphi|A_{(H)}$ extends to a $G$-map $\vphi' := (\vphi_0', \vphi_1' \eta) : X^{(H)} \longra CE$ with $\vphi'(B)=\{0\}$.
As $X^{(H)}$ is closed in $X$ by Lemma~\ref{lem:closed}, by pasting lemma \cite[18.3]{Munkres}, $\vphi$ and $\vphi'$ unite to a $G$-map $\vphi'': A \cup X^{(H)} \longra CE$.

Again as above, the restriction of this new function has coordinates
\[
\vphi''| A' = (\vphi''_0, \vphi''_1) : A' := (A-Z) \cup X_{(H)} \longra E \x (0,1].
\]
Since $H$ is a compact subgroup of the Lie group $G$, the orbit $H\bs G \in G\text{-ANR}(G\text{-}\cM)$ by Palais' slice theorem \cite[2.3.1]{Palais}.
Then $E \in G\text{-ANE}(G\text{-}\cM)$ by Theorem~\ref{thm:Banakh_generalized}.
So $\vphi''_0: A' \longra E$ extends to a $G$-map $\Phi'_0: U \longra E$ on a $G$-neighborhood $U$ of the \emph{closed} $G$-subset $A'$ in the $G$-metrizable space $X' := X-(Z\cup B)$.
Indeed, by Lemma~\ref{lem:closed}, the frontier $\bdry_X(X_{(H)}) := \ol{X_{(H)}} - X_{(H)} \subset B$ so $\bdry_{X-B}(X_{(H)}) = \emptyset$.
Define a $G$-map
\[
\eta': X' \longra [0,1] ~;~ x \longmapsto \frac{d(x,X'-U)}{d(x,A') + d(x,X'-U)}
\]
with $\eta'(X'-U)=\{0\}$ and $\eta'(A')=\{1\}$.
Then $\vphi''_1: A' \longra (0,1]$ extends to a map $\vphi''_1 \eta': X' \longra [0,1]$.
So $\vphi''|A'$ extends to a $G$-map $\Phi' := (\Phi_0', \vphi''_1 \eta') : X' \longra CE$ with $(\Phi')\inv\{0\} = X'-U$.
Extend $\Phi'$ by zero to a $G$-map $\Phi: X' \cup Z \cup B \longra CE$.
\end{proof}

Recall that any $G$-map $f: X \longra Y$ satisfies $G_x \leqslant G_{fx}$ for all $x \in X$.
Furthermore, if $G_x = G_{fx}$ for all $x \in X$, the $G$-equivariant map $f$ is called \textbf{$G$-isovariant}.

\begin{thm}\label{thm:extensor}
Let $G$ be a Lie group.
Let $\cF \subseteq \cpt(G)$ have no conjugate elements.
Then $\bE_\cF G = \bigbcoast_{H \in \cF} \bE(H\bs G)$ is an \textbf{isovariant absolute $G$-extensor} for $(G,\cF)\text{-}\cM$:  for any closed $G$-subset $A$ of any member $X$ of the class $(G,\cF)\text{-}\cM$, any isovariant $G$-map $\vphi: A \longra \bE_\cF G$ extends to an isovariant $G$-map $\Phi: X \longra \bE_\cF G$.
\end{thm}

Differently from Palais--Bredon's construction of an isovariant $\cF$-classifying space, S~Ageev asserted that $(C(\cF \bs G))^{\aleph_0}$ is an isovariant absolute $G$-extensor for $(G,\cF)\text{-}\cM$, if $G$ is a compact Hausdorff group and $\cF \subseteq \lrg(G)$ need not be finite \cite[3.2]{Ageev4}.

\begin{proof}
Write $E_H := \bE(H\bs G)$ and denote coordinates $\vphi = (\vphi_H: A \longra C E_H)_{H \in \cF}$.
Let $H \in \cF-\{G\}$.
Since $\vphi$ is isovariant, $\bigbcoast$ is our unrestricted isovariant join (\ref{defn:Khan_join}), and $\cF$ has no conjugate elements, note for all $a \in A_{(H)}$ that
\[
(G) \neq (H) = (G_a) = (G_{\vphi a}) = (G_{\vphi_H a}).
\]
So $0 \notin \vphi_H(A_{(H)})$.
By Lemma~\ref{lem:PB_cone}, $\vphi_H$ extends to a $G$-map $\Phi_H : X \longra C E_H$ with $0 \notin \Phi_H(X_{(H)})$.
If $G \in \cF$ then $\vphi_G: A \longra CE_G = [0,1]$, where $0 \notin \vphi_G(A_G)$ since $\{0\}^I \notin \vphi(A_G)$, extends to a $G$-map $\Phi_G: X \longra [0,1]$ with $0 \notin \Phi_G(X_G)$, by \cite[Satz~3]{Tietze} via orbit spaces to $(0,1] \in \text{AE}(\cM)$ then $[0,1]$; this is if $G$ is compact and $X$ has a $G$-fixed point.
So the $G$-map $\Phi := (\Phi_H)_{H \in \cF}: X \longra \bE_\cF G$ is isovariant.
\end{proof}

Consequently, we obtain the desired corollary which is the first half of uniqueness.

\begin{proof}[Proof of Theorem~\ref{thm:homotopy}]
Assume a $G$-homeomorphism $\psi: f^*(\bE_\cF G) \longra g^*(\bE_\cF G)$ satisfying $\psi/G = \id_B$.
Note $X := f^*(\bE_\cF G) \x [0,1]$ has orbit space $X/G = B \times [0,1]$.
On the closed $G$-subset $A := f^*(\bE_\cF G) \x \{0,1\}$ of $X$, define the isovariant $G$-map
\[
\vphi: A \longra \bE_\cF G ~;~ (b,e,s) \longmapsto
\begin{cases}
e & \text{if } s=0\\
\psi_1(b,e) & \text{if } s=1
\end{cases}
\]
with the pullback $f^*(\bE_\cF G) := \{ (b,e) \in B \x \bE_\cF G ~|~ f(b)=eG \}$ and $\psi = (\id_B, \psi_1)$.
Note $(\vphi/G)(b,0) = f(b)$ and $(\vphi/G)(b,1) = \psi_1(b,e)G = g(b)$.
Therefore, we conclude the existence of a stratified homotopy $\Phi/G: B \x [0,1] \longra \bB_\cF G$ from $f$ to $g$ by Theorem~\ref{thm:extensor}, once we verify that $X$ is $G$-metrizable, as $X$ has isotropy in $(\cF)$.

Since $H\bs G \in G\text{-}\cM$ for any $H \in \cF$, by Theorem~\ref{thm:Banakh_generalized}, the Milnor join $\bE(H \bs G) \in G\text{-}\cM$.
The induced metric (\ref{rem:cone_metric}) on its coarse cone is $G$-metrizable.
The Lie group $G$ has only countably many conjugacy classes of compact subgroups, by \cite[Corollary~3.9]{Khan_Lie}.
Then the countable product $\prod_{H \in \cF} C\bE(H\bs G)$ has an induced metric \cite[20.5]{Munkres}, whose formula is $G$-invariant.
So $\bE_\cF G \in G\text{-}\cM$.
Therefore, since $B \in \cM$, the subproduct $f^*(\bE_\cF G) \subset B \x \bE_\cF G \in G\text{-}\cM$, hence $X$ is also a member.
\end{proof}

\section{The classifying property: uniqueness, II}

The following Covering Homotopy Theorem is a nontrivial result on product structures for Hausdorff $B$.
In the free case, it is \cite[11.3]{Steenrod_book} if $B$ is normal Lindel\"of and locally compact, and more generally \cite[4:9.8]{Husemoller} if $B$ is paracompact.\footnote{Paracompact Hausdorff $B$ admit a product-structure theorem for microbundles \cite[3.1]{Milnor_microbundles}.}
If $G$ is a compact Lie group, the result generalizes \cite[2.4.1]{Palais_book} if $B$ is second-countable locally compact, and more generally \cite[II:7.1]{Bredon_TG} if $B$ is hereditarily paracompact.

\begin{thm}\label{thm:product}
Let $X$ be a Tikhonov space with Palais action of a Lie group $G$.
Suppose the orbit map is $p: X \longra B \x [0,1]$ for some hereditarily paracompact Hausdorff space $B$.
Assume $(G_x) = (G_y)$ if $\proj_B(p x) = \proj_B(p y)$.
Then $X$ is $G$-homeomorphic over the identity $\id_{B \x [0,1]}$ to the product space $p\inv(B \x \{0\}) \x [0,1]$.
\end{thm}

Our ensuing proof applies and extends Palais--Bredon's argument \cite[II:7.1]{Bredon_TG}.

\begin{lem}\label{lem:StepA}
Let $(y,t) \in B \x [0,1]$.
Then $p\inv(U \x [a,b])$ is $G$-homeomorphic over $\id$ to $p\inv(U \x a) \x [a,b]$ for some neighborhoods $U$ of $y$ in $B$ and $[a,b]$ of $t$ in $[0,1]$.
Furthermore, the $G$-homeomorphism restricts to $\id: p\inv(U \x a) \longra p\inv(U \x a) \x a$.
\end{lem}

We modify \cite[Proof~II:7.1A]{Bredon_TG} to include all strata and to exclude induction.

\begin{proof}
Let $x \in p\inv(y,t)$.
Since $X$ is Tikhonov with Palais $G$-action, by Palais' slice theorem \cite[2.3.1]{Palais}, there exists a $G_x$-slice $S$ at $x$ in $X$.
The tube $S G$ is open in $X$, hence its image $p(S G)$ is open in the orbit space $X/G = B \x [0,1]$.
By the tube lemma, there are neighborhoods $U$ of $y$ in $B$ and $[a,b]$ of $t$ in $[0,1]$ such that $U \x [a,b] \subseteq p(S G)$.
We may assume equality by reassigning $S$ as $S \cap p\inv(U \x [a,b])$.

Since $G_x$ is a compact Lie group, $S/G_x = SG/G = U \x [a,b]$, and $U$ is hereditarily paracompact, by \cite[Theorem~II:7.1]{Bredon_TG}, $S$ is $G_x$-homeomorphic over $\id_{U \x [a,b]}$ to the product $T \x [a,b]$ with $T := S \cap p\inv(U \x a)$ and $[a,b]$ trivial $G_x$-action.
Note
\[
p\inv(U \x [a,b]) = SG = S \x_{G_x} G \approx (T  \x_{G_x} G) \x [a,b] = p\inv(U \x a) \x [a,b].\qedhere
\]
\end{proof}

\begin{lem}\label{lem:StepB}
Let $y \in B$.
The preimage $p\inv(U \x [0,1])$ is $G$-homeomorphic over $\id_{U \x [0,1]}$ to the product $p\inv(U\x 0) \x [0,1]$ for some neighborhood $U$ of $y$ in $B$.
Furthermore, the $G$-homeomorphism restricts to $\id: p\inv(U \x 0) \longra p\inv(U \x 0) \x 0$.
\end{lem}

Our argument reexplains \cite[Proof~II:7.1B]{Bredon_TG} but now includes all the strata.

\begin{proof}
For each $t \in [0,1]$, by Lemma~\ref{lem:StepA}, there exist a neighborhood $U_t$ of $y$ in $B$, a neighborhood $[a_t,b_t]$ of $t$ in $[0,1]$, and a $G$-homeomorphism $\phi_t$ over the identity:
\[
\phi_t : p\inv(U_t \x [a_t,b_t]) \longra p\inv(U_t \x a_t) \x [a_t,b_t] \quad\text{with}\quad \phi_t|p\inv(U_t \x a_t) = \id.
\]
Since $[0,1]$ is compact, there is a finite subset $F \subset [0,1]$ with $(0,1) = \bigcup_{t \in F} (a_t,b_t)$.
Define $U := \bigcap_{t \in F} U_t$, a neighborhood of $y$ in $B$.
By Lebesgue's number lemma, there is $n \in \N$ such that each $\left[\frac{i}{n},\frac{i+1}{n}\right] \subseteq [a_{t_i},b_{t_i}]$ for some $t_i \in F$.
Thus we obtain
\begin{gather*}
\vphi_i : p\inv(U \x \left[{\textstyle\frac{i}{n}},{\textstyle\frac{i+1}{n}}\right]) \longra p\inv(U \x 0) \x \left[{\textstyle\frac{i}{n}},{\textstyle\frac{i+1}{n}}\right]\\
\vphi_i| = \phi_{t_0} \circ \cdots \circ \phi_{t_{i-1}} : p\inv(U \x {\textstyle\frac{i}{n}}) \longra p\inv(U \x 0).
\end{gather*}
Then $\psi := \vphi_0 \cup \cdots \cup \vphi_{n-1}: p\inv(U \x [0,1]) \longra p\inv(U \x 0) \x [0,1]$ as desired.
\end{proof}

We shall avoid Bredon's transfinite induction by a Milnor-style replacement trick.
Our statement is more generally in terms of predicates $\Pi$ (that is, unary relations).

\begin{prop}\label{prop:countable}
Let $(B,\cT)$ be a normal Hausdorff space.
Let $\Pi \subseteq \cT$ be preserved under all open subsets and all disjoint unions.
Suppose $\cU \subseteq \Pi$ for a locally finite open cover $\cU$ of $B$.
Then $\cV \subseteq \Pi$ for a \emph{countable} locally finite open cover $\cV$ of $B$.
\end{prop}

For local trivializations, \cite[Hilfsatz~2]{tomDieck_numerable} \cite[4:12.1]{Husemoller} work.
Originally, Milnor proved it for $(B,\cT)$ paracompact and no input $\cU$ \cite[p25--26]{Milnor_notes} \cite[5.9]{MilnorStasheff}.

\begin{proof}
Since $B \in T_4$, by Dieudonn\'e's shrinking lemma \cite[Th\'eor\`eme 6]{Dieudonne} and Urysohn's lemma \cite[Satz~25]{Urysohn}, it follows as noted in \cite[Proposition~2]{Michael} that $\cU$ admits a subordinate partition of unity $\{t_U: B \longra [0,1] \}_{U \in \cU}$ with the \emph{same} index set.
That is, $t_U$ are continuous with $\{t_U > 0\} \subseteq U$ and $\sum_{U \in \cU} t_U = 1$.

We may assume $\cU$ is infinite.
For each nonempty finite $F \subset \cU$, define the~function
\[
q_F ~:=~ \max\left\{ ~0, ~\min_{U \in F} t_U - \max_{U \notin F} t_U ~\right\} ~:~ B \longra [0,1].
\]
Observe that $q_F$ is continuous, since the $\min$ is over the finite set $F$ and each $x \in B$ admits a neighborhood $N$ meeting only finitely many elements of $\cU$ hence of $\cU-F$.
Then its support $V_F := \{ q_F > 0 \}$ is open.
So $V_F \in \Pi$ since $V_F \subseteq U$ for some $U \in F$.

Let $E \neq F$ be finite subsets of $\cU$ with $\card\,E = \card\,F$.  We show $V_E \cap V_F = \emptyset$.
As $E$ and $F$ are distinct of same cardinality, there are $O \in E - F$ and $P \in F - E$.
If $x \in V_E$ then $\DS t_O(x) \geqslant \min_{U \in E} t_U(x) > \max_{U \notin E} t_U(x) \geqslant t_P(x)$.
Similarly, if $x \in V_F$ then $\DS t_P(x) \geqslant \min_{U \in F} t_U(x) > \max_{U \notin F} t_U(x) \geqslant t_O(x)$.
Thus $V_E \cap V_F = \emptyset$, as $t_O > t_P > t_O$ on $V_E \cap V_F$.
Then the union $V_n := \bigcup \{ V_F ~|~ \card\,F = n \}$ is disjoint.
Therefore $V_n \in \Pi$.

Define $\cV := \{ V_n ~|~ n \in \Z_{>0} \}$, a countable collection of open sets in $B$.
We show $\cV$ is a cover of $B$.
Let $x \in B$.
Since $\sum_{U \in \cU} t_U(x) = 1$, the set $S := \{ U \in \cU ~|~ t_U(x) > 0 \}$ is nonempty finite.
Then $\DS q_S(x) = \min_{U \in S} t_U(x) > 0$.
So $x \in V_S \subseteq V_{\card\, S}$.
Lastly, we show $\cV$ is locally finite.
As $\cU$ is locally finite, there is a neighborhood $N$ of $x$ in $B$ with $R := \{ U \in \cU ~|~ N \cap U \neq \emptyset \}$ finite.
Suppose $N \cap V_F \neq \emptyset$.
Then $q_F(y) > 0$ for some $y \in N$.
So $t_U(y)>0$ hence $y \in U$ for all $U \in F$.
Thus $F \subseteq R$.
That is, $\{F ~|~ N \cap V_F \neq \emptyset \} \subseteq 2^R$, which is a finite set.
Therefore $\cV$ is locally finite.
\end{proof}

We simplify \cite[Proof~II:7.1C]{Bredon_TG} to include all strata and no infinite ordinals.
For ease of reading, we drop the $\x 0$ for $p$-preimages that occur from Lemma~\ref{lem:StepB}.
The key idea is to construct a $G$-isotopy $\delta$ on the overlap for continuous transition.

\begin{proof}[Proof of Theorem~\ref{thm:product}]
Since $B$ is paracompact Hausdorff hence normal by \cite[Th\'eor\`eme~1]{Dieudonne}, by Lemma~\ref{lem:StepB} and Proposition~\ref{prop:countable}, we obtain a countable locally finite open cover $\cV = \{ V_n \}_{n>0}$ of $B$ and $G$-homeomorphisms $\psi_n: p\inv(V_n \x [0,1]) \longra p\inv(V_n) \x [0,1]$ over $\id_{V_n \x [0,1]}$ restricting to $\id: p\inv(V_n \x 0) \longra p\inv(V_n) \x 0$.

Consider the open sets $U_n := \bigcup_{i < n} V_i$ in $B$, with $U_1 = \emptyset$.
Since $\cV$ is locally finite, it suffices to recursively define similar $G$-homeomorphisms $\phi_n : p\inv(U_n \x [0,1]) \longra p\inv(U_n) \x [0,1]$ such that $\phi_{n+1}| = \phi_n|$ over $(U_n - V_n) \x [0,1]$.
Define $\phi_1 = \id_\emptyset$.

Assume $\phi_n$ is defined.
Shorten $U := U_n$ and $V := V_n$, so $U_{n+1} = U \cup V$.
Write
\[
\eps ~:=~ \proj_X \circ \phi_n \circ \psi_n\inv| ~:~ p\inv(U \cap V) \x [0,1] \longra p\inv(U \cap V).
\]
As $U \cup V$ is paracompact so normal, disjoint closed sets $V-U$ and $U-V$ have disjoint \emph{closed} neighborhoods $C_0$ and $C_1$ in $U \cup V$.
By Urysohn's lemma \cite[Satz~25]{Urysohn}, there is a map $f: U \cup V \longra [0,1]$ with $f(C_i) = \{i\}$.
Define a $G$-homeomorphism
\[
\delta : p\inv(U \cap V) \x [0,1] \longra p\inv(U \cap V) \x [0,1] ~;~ (x,t) \longmapsto (\eps(x,f(px)t), t)
\]
with inverse $\delta\inv(y,t) = (\proj_X \psi_n \phi_n\inv (y,f(py)t), t)$.
Note if $px \in V-U$ then $\delta(x,t) = (x,t)$, and if $px \in U-V$ then $\delta(x,t) = \phi_n \psi_n\inv (x,t)$.
Define the $G$-bijection
\[
\phi_{n+1} ~:=~ 
\left\{\begin{aligned}
\phi_n & \text{ on } U - V\\
\delta \circ \psi_n & \text{ on } U \cap V\\
\psi_n & \text{ on } V - U
\end{aligned}\right\}
~:~ p\inv(U_{n+1} \x [0,1]) \longra p\inv(U_{n+1}) \x [0,1].
\]
By the pasting lemma, $\phi_{n+1}$ is continuous as $f$ is constant on each neighborhood $C_i$; similarly for the formula of $\phi_{n+1}\inv$.
Then $\phi_{n+1}$ is obtained.
Induction is complete.
\end{proof}

We conclude with a summary of our uniqueness results, now inclusive of noncompact $G$.

\begin{cor}\label{cor:unique2}
Let $G$ be an arbitrary Lie group.
Let $\cF \subseteq \cpt(G)$ with no conjugate elements.
Let $B$ be an $(\cF)$-filtered metrizable space.
Two maps $f, g: B \longra \bB_\cF G$ are stratified-homotopic if and only if $f^*(\bE_\cF G)$ and $g^*(\bE_\cF G)$ are $G$-homeomorphic over $\id_B$.
\end{cor}

\begin{proof}
The reverse direction is Theorem~\ref{thm:homotopy}.
For the forward direction, let $h: B \x [0,1] \longra \bB_\cF G$ be a stratified homotopy from $f$ to $g$.
Write $X := h^*(\bE_\cF G)$, which is a $G$-metrizable Palais $G$-space as shown in Proof~\ref{thm:homotopy}.
Hence $X$ is Tikhonov.
The metrizable space $B$ is hereditarily paracompact \cite{Stone1}.
Write $p: X \longra B \x [0,1]$.
As $h$ is stratified, $(G_x) = (G_y)$ if $\proj_B(px) = \proj_B(py)$.
By Theorem~\ref{thm:product}, there is a $G$-homeomorphism $X \longra f^*(\bB_\cF G) \x [0,1]$ over $\id_{B \x [0,1]}$.
It restricts to a $G$-homeomorphism $g^*(\bB_\cF G) = p\inv(B \x 1) \longra f^*(\bB_\cF G) \x 1$.
\end{proof}

\begin{rem}[Baum--Connes--Higson]
For $G$ a locally compact Hausdorff group, cardinal $\kappa = \aleph_0$, family $\cF = \cpt$, and $B$ a metrizable space, a weaker variation of Corollary~\ref{cor:unique2} is sketched in \cite[Appendix~3]{BCH}.
Their correspondence is between ordinary homotopy classes of maps $B \longra B_\cpt^{\aleph_0} G$ and their so-called `homotopy' classes of Palais-proper $G$-spaces over $B$, which I instead would call \emph{concordance} classes.
\end{rem}

Here is our full classification generalizing \cite[2.6.2, 2.7.10]{Palais_book} \cite[II:9.7]{Bredon_TG}.
We allow noncompact $G$, infinite $\cF$, and infinite $\dim(B)$; thus, we fulfill and exceed Palais' ambition \cite[\S4.5]{Palais}.

\begin{thm}\label{thm:unique}
Let $G$ be an arbitrary Lie group.
Let $\cF \subseteq \cpt(G)$ with no conjugate elements.
Let $B$ be an $(\cF)$-filtered metrizable space.
Taking pullback of $\bE_\cF G$ is a bijection from stratified-homotopy classes of stratified maps $B \longra \bB_\cF G$ to isomorphism classes of metrizable spaces that are equipped with Palais $G$-action, isotropy conjugate to members of $\cF$, and orbit space $B$.
\end{thm}

\begin{proof}
Well-definition and injectivity of this correspondence are Corollary~\ref{cor:unique2}.
To show surjectivity, let $X$ be a metrizable space with Palais $G$-action, isotropy conjugate into $\cF$, and $B = X/G$ (or a stratified homeomorphism).
By Antonyan--deNeymet \cite[Theorem~B]{AdN}, $X$ admits a $G$-invariant metric.
As $X \in (G,\cF)\text{-}\cM$, by Theorem~\ref{thm:extensor}, there is an isovariant $G$-map $\Phi: X \longra \bE_\cF G$.
The induced map $X \longra (\Phi/G)^*(\bE_\cF G)$ is a $G$-homeomorphism over $\id_B$ \cite[2.5]{Khan_Lie}.
\end{proof}

\subsection*{Acknowledgements}

I am grateful to Dennis Burke for discussion on orthocompactness and to Klaas Pieter Hart for helping me to locate Lindel\"of's relevant work.
The referee's request inspired Remark~\ref{rem:comparison}.

\bibliographystyle{elsarticle-harv}
\bibliography{CardinalClassifyingSpace}

\end{document}